\documentclass{siamart0216}

\newsiamremark{remark}{Remark}

\usepackage{amsfonts,amsmath,amssymb}
\usepackage{mathrsfs,mathtools,stmaryrd}
\usepackage{enumerate}
\usepackage{xspace} 
\usepackage{hyperref}

\newcommand{\bv}{\mathbf{v}}

\newcommand{\R}{\mathbb{R}}
\newcommand{\polT}{\mathbb{T}}
\newcommand{\frakt}{\mathfrak{t}}
\DeclareMathOperator{\suc}{succ}

\newcommand{\calS}{\mathcal{S}}

\newcommand{\bh}{\mathbf{h}}
\newcommand{\be}{\mathbf{e}}
\newcommand{\GRAD}{\nabla}
\DeclareMathOperator{\DIV}{div}
\newcommand{\diff}{\, \mbox{\rm d}}

\DeclareMathOperator{\diam}{diam}
\newcommand{\polN}{{\mathbb{N}}}
\newcommand{\bxi}{{\boldsymbol{\xi}}}
\newcommand{\Poinc}{{\mathtt{P}_\ell}}
\newcommand{\Bog}{{\mathtt{B}_\ell}}
\newcommand{\eps}{\varepsilon}

\DeclareMathOperator{\supp}{supp}

\DeclareMathOperator{\tr}{tr}

\newcommand{\ipd}{\lrcorner}
\renewcommand{\imath}{{\mathfrak{i}}}

\newcommand{\TheTitle}{Estimation of the continuity constants for Bogovski\u{\i} and regularized Poincar\'e integral operators}
\newcommand{\ShortTitle}{Constants for Bogovski{\u{\i}} and Poincar\'e operators}
\newcommand{\TheAuthors}{J.~Guzm\'an, A. J.~Salgado}

\headers{\ShortTitle}{\TheAuthors}

\title{{\TheTitle}\thanks{JG has been partially supported by NSF grant {DMS--1913083}. AJS has been partially supported by NSF grant DMS--1720213.}}

\author{
  Johnny Guzm\'an\thanks{Division of Applied Mathematics, Brown University, Providence, RI, 02912, USA.
    (\email{johnny\_guzman@brown.edu})}
  \and
  Abner J.~Salgado\thanks{Department of Mathematics, University of Tennessee, Knoxville, TN 37996, USA.
    (\email{asalgad1@utk.edu}, \url{http://www.math.utk.edu/\string~abnersg})}
}

\ifpdf
\hypersetup{
  pdftitle={\TheTitle},
  pdfauthor={\TheAuthors}
}
\fi

\begin{document}

\maketitle

\begin{center}
  \emph{We dedicate this work to the memory of our friend Francisco Javier Sayas}
  
  \emph{Rest in peace Pancho!}
\end{center}

\begin{abstract}
We study the dependence of the continuity constants for the regularized Poincar\'e and Bogovski\u{\i} integral operators acting on differential forms defined on a domain $\Omega$ of $\R^n$. We, in particular, study the dependence of such constants on certain geometric characteristics of the domain when these operators are considered as mappings from (a subset of) $L^2(\Omega,\Lambda^\ell)$ to $H^1(\Omega,\Lambda^{\ell-1})$, $\ell \in \{1, \ldots, n\}$. For domains $\Omega$ that are star shaped with respect to a ball $B$ we study the dependence of the constants on the ratio $\diam(\Omega)/\diam(B)$. A program on how to develop estimates for higher order Sobolev norms is presented. The results are extended to certain classes of unions of star shaped domains.
\end{abstract}



\section{Introduction}
\label{sec:Intro}

A fundamental result in the analysis of models of incompressible fluids is the existence of a right inverse for the divergence operator. Let $\Omega \subset \R^n$, with $n \in \polN$, be a bounded domain with Lipschitz boundary, and $u \in L^2(\Omega)$ be such that $\int_\Omega u \diff x= 0$. Then there is a vector field $\bv \in H^1_0(\Omega,\R^n)$ such that
\[
  \DIV \bv = u, \qquad \| \GRAD \bv \|_{L^2(\Omega, \R^{n\times n})} \leq C \| u \|_{L^2(\Omega)},
\]
where the constant $C$ depends on $\Omega$, but not on $u$; see Section~\ref{sec:Notation} for notation. While this problem has been studied in several sources, and from different points of view; see \cite{MR0227584,MR0155092,MR1846644} for a very incomplete list of references and the introduction to \cite{MR3618122} for a nice historical account, we are interested here in \cite{MR553920,MR631691}, where the function $\bv$  is explicitly constructed. First, on domains that are star shaped with respect to a ball, the function is constructed by means of a regularized version of a path integral. This regularization is necessary, as integrals of $u$ along paths may not be well defined. In passing, the author mentions that the constant in the norm estimate depends on the ratio between the diameter of the domain and that of the ball. Then, for domains that can be represented as a finite union of star shaped domains, the construction is attained via a partition of unity argument. Further mapping properties of this operator have been discussed, for instance, in \cite{MR2731700,MR1880723,MR2643399,MR2548872}, and we refer the reader, again, to \cite{MR3618122} for a rather recent overview.

It is well known that vector fields and the operators of vector calculus, like the divergence, are nothing but particular cases of differential forms on the domain $\Omega$ and the exterior derivative \cite{MR3837152,MR0344035}. Thus, it is only natural to pose the question about the existence of a right inverse for the exterior derivative or, in the language of differential forms, to find conditions that guarantee that a closed form is exact. An explicit solution to this problem is presented in \cite{MR0344035} and it uses, once again, a path integral for its definition. As the classical theory is only concerned with differential forms with at least continuous coefficients, integrals over paths are meaningful. The situation is rather different if one wishes to deal with differential forms having coefficients that are not smooth, as this construction may not be valid anymore. To the best of our knowledge, a regularized version of the solution in \cite{MR0344035} was first presented in \cite{MR1241286}, see also \cite[Appendix]{MR2166752}. It is also shown \cite[Proposition 4.1]{MR1241286} that, if $\Omega$ is convex and $u \in L^2(\Omega, \Lambda^\ell)$, with $\ell \in \{1, \ldots,n\}$, is such that $\diff u = 0$, then there is $v \in H^1(\Omega, \Lambda^{\ell-1})$ such that
\[
  \diff v = u, \qquad \| v \|_{L^2(\Omega, \Lambda^{\ell-1})} + \diam(\Omega) \| \GRAD v \|_{L^2(\Omega,\Lambda^{\ell-1})^n} \leq C \diam(\Omega)^n \| u \|_{L^2(\Omega,\Lambda^\ell)},
\]
where the constant $C$ is only dimension dependent. From this a Poincar\'e--Sobolev inequality is obtained \cite[Corollary 4.2]{MR1241286}. We refer the reader to \cite{MR1857238,MR2834164,MR2271946,MR2552910} for further estimates for this operator.

A remarkable property of the operator constructed by Bogovski\u{\i} is the fact that the vector field $\bv$ has a vanishing trace. In fact, in the case of a domain that is star shaped with respect to a ball, the value of $\bv$ at a point depends only on the convex hull of the point and the ball. Therefore, $\bv$ is supported in $\bar\Omega$. This local property is rather unusual for integral operators, and it does not hold for the operators constructed in \cite{MR1241286}. In \cite{MR2425010} it was observed that by taking adjoints the locality is recovered. By conjugating with the Hodge star operator the authors were able to construct, for every $\ell \in \{1, \ldots, n\}$, two integral operators that preserve the locality properties and they proceeded to show several mapping properties for them. Further mapping properties of these operators were investigated  Costabel and McIntosh \cite{MR2609313}. In particular, they prove that these operators are bounded in various Sobolev norms when the domain is star shaped with respect to a ball. However, they do not track how the constants explicitly depend on the geometry of the domain. We call these operators the Bogovski\u{\i}--type and Poincar\'e--type integral operators, and it is our goal here to study the dependence of the continuity constants of these operators on some geometric characteristics of the domain. 

On the other hand Dur\'an, in \cite{MR3086804}, gives explicit bounds for the constant for the Bogovski\u{\i} operator in  the $H^1$--norm. These estimates improve on those given by Galdi \cite{galdi2011introduction}.  We will adopt the ideas of \cite{MR3086804}, where the operator from \cite{MR553920,MR631691} is considered, to both operators and the whole range of orders for differential forms. We must immediately mention that, since this technique heavily uses properties of the Fourier transform, all of the results that we obtain are for $L^2$--based spaces.  Bounds for the  Bogovski\u{\i}--type and Poincar\'e--type integral operators are needed in finite element methods (FEM); see for example \cite{morin2003local, farrell2020robust, chaumont2020stable, melenk2020commuting,boffi2011discrete, ern2015polynomial, braess2009equilibrated}. For simple geometries arising in FEM like a simplex one can prove the estimates by mapping to a reference simplex. However, for more complicated geometries arising in FEM (like curved elements or a patch of simplices) it might be useful to have results like the ones described in this paper.

Our presentation is organized as follows: In Section \ref{sec:Notation} we provide preliminaries. In Section \ref{sec:Poincare} we focus on the Poincar\'e--type operator. We mainly focus on $H^1$--estimates but we also show how to get a bound for the constant for $H^2$--norm. In the following section we obtain bounds for the Bogovski\u{\i}--type operator.  Finally, in Section~\ref{sec:LipschitzDomains} we use the results from the previous sections to give bounds for the constants if one has a chain of star shaped domains.

\section{Notation and preliminaries}
\label{sec:Notation}

Let us begin by presenting the notation that we will follow, together with some preliminary facts that shall be repeatedly used during the course of our presentation. During the course of our discussion $n \in \polN$ will indicate the spatial dimension. $\Omega \subset \R^n$ indicates a bounded domain with at least Lipschitz boundary. If we require additional conditions on $\Omega$ these will be indicated explicitly.
% We recall that if $\Omega$ has Lipschitz boundary the exterior normal, denoted by $\mathbf{n}$, is almost everywhere defined. 
For any bounded, measurable, domain $E \subset \R^n$ we will indicate by $\diam(E)$ its diameter and by $|E|$ its Lebesgue measure. We will follow standard notation and definitions for real valued smoothness spaces over $\Omega$. We will make use of the Fourier transform and we refer the reader to \cite[Section 2.2.4]{MR3243734} for relevant results.

By $C$ we will indicate an nonessential constant whose value may change from line to line. If we wish to indicate explicitly that this constant depends on certain parameters, say $\alpha,\beta,\gamma$, we denote this by $C(\alpha,\beta,\gamma)$. By nonessential in this work we will mean that the constant does not depend on $\Omega$ or its geometric characteristics.

Let $D \subset \R^n$ be a bounded domain that is star shaped with respect to a ball $B \subset D$. By this we mean that every for every $y \in D$ the convex hull of $B \cup \{y\}$ is contained in $D$. It is known that \cite[Lemma 3.2.4]{MR1622690} every bounded Lipschitz domain can be represented as a finite union of domains that are star shaped with respect to a ball. In addition, it can be shown that there is $\theta \in C^\infty_0(B)$ such that
\begin{equation}
\label{eq:PropertiesTheta}
  \int \theta \diff x= 1, \quad \| \partial^\alpha \theta \|_{L^1(D)} \leq \frac{C(n,\alpha)}{\diam(B)^{|\alpha|}}, \quad \| \partial^\alpha \theta \|_{L^\infty(D)} \leq \frac{C(n,\alpha)}{\diam(B)^{n+|\alpha|}}. 
\end{equation}

\subsection{Differential forms on domains}
\label{sub:DiffForms}

For $\ell \in \{0, \ldots, n\}$ we denote by $\Lambda^\ell$ the vector space of exterior $\ell$--forms, that is the space of skew--symmetric $\ell$--linear functions on $(\R^n)^\ell$. In this notation $\Lambda^0 = \R$, and $\Lambda^1$ is the dual of $\R^n$. For $\omega_\ell \in \Lambda^\ell$ and $\omega_k \in \Lambda^k$ their exterior product is $\omega_\ell \wedge \omega_k \in \Lambda^{\ell+k}$. We recall that $\omega_\ell \wedge \omega_k = (-1)^{k \ell} \omega_k \wedge \omega_\ell$. Let $\{\be_i\}_{i=1}^n \subset \R^n$ be the canonical basis, and $\{\be^i\}_{i=1}^n \subset \Lambda^1$ its dual basis, then any $\omega_\ell \in \Lambda^\ell$ can be uniquely represented by
\[
  \omega = \sum_{I} \omega_I \be^I,
\]
where $\omega_I \in \R$, the sum runs over all ordered $\ell$--tuples of indices: $I = (i_1, \ldots, i_\ell) \subset \polN^\ell$, $1 \leq i_1 < i_2 < \ldots < i_\ell \leq n$, and
\[
  \be^I = \be^{i_1} \wedge \ldots \wedge \be^{i_\ell}.
\]
Whenever $I$ is such an ordered $\ell$--tuple of indices we will denote, for $m \in \{1,\ldots,\ell\}$, $\hat I_m = (i_1, \ldots, i_{m-1}, i_{m+1}, \ldots, i_\ell) \subset \polN^{\ell-1}$, that is we suppress the index tagged by $m$. Finally, to describe one result we shall need to make use of the Hodge star operator $\star$. For $\ell \in \{0, \ldots, n\}$ this is a mapping $\star: \Lambda^\ell \to \Lambda^{n-\ell}$ defined by
\[
  \star \left( \sum_{I} \omega_I \be^I\right) = \sum_{I}  (-1)^{\sigma(I)}\omega_I \be^{I^c},
\]
where $I^c = \{1, \ldots, n\}\setminus I$, and $\sigma(I) = 0$ if $I \sqcup I^c$ forms an even permutation, and $\sigma(I) = 1$ otherwise. Notice that this induces an inner product on $\Lambda^\ell$
\[
  \langle u, v \rangle \be^1 \wedge \ldots \wedge \be^n = u \wedge \star v, \quad |u|^2 = \langle u, u \rangle \in \R.
\]

Throughout our work we will be concerned with differential $\ell$--forms on $\Omega$, that is functions on $\Omega$ that have values in $\Lambda^\ell$. Thus, for instance, if $p \in [1,\infty]$ the space of differential $\ell$--forms with components belonging (in some coordinate system) to $L^p(\Omega)$ is denoted by $L^p(\Omega,\Lambda^\ell)$. If $\omega$ is a differential $\ell$--form on $\Omega$, that is differentiable at $x \in \Omega$, its derivative is 
\[
  D\omega(x) : \R^n \to \Lambda^\ell.
\]
More specifically, for $\bh \in \R^n$, we have the definition
\[
  D\omega(x) \bh=\lim_{t \rightarrow 0} \frac{\omega(x+t \bh) -\omega(x)}{t},
\]
where the limit is taken in $\Lambda^\ell$.
%Thus, if $\bh \in \R^n$ $D\omega(x) \bh \in \Lambda^\ell$. 
The exterior derivative $\diff \omega(x)$ is an $(\ell+1)$--form defined by
\[
  \diff\omega(x;\bxi_1, \ldots, \bxi_{\ell+1}) = \sum_{i=1}^{\ell+1} (-1)^{i-1} \left[ D\omega(x) \bxi_i\right](\bxi_1, \ldots,\hat{\bxi}_i, \ldots, \bxi_{\ell+1}),
\]
where, for $i \in \{1, \ldots, \ell+1\}$, $\bxi_i \in\R^n$. The coordinate functions $x_1, \ldots, x_n$ are considered differential forms of degree zero. The one forms $\{\diff x_i \}_{i=1}^n$ are constant functions from $\Omega$ into $\Lambda^1$
\[
  \diff x_i(x) = \be^i.
\]
Thus, every $u \in L^p(\Omega, \Lambda^\ell)$ can be uniquely represented as
\[
  u(x) = \sum_I u_I(x) \diff x_I, \qquad u_I \in L^p(\Omega),
\]
where $\diff x_I$ and the set of indices $I$ have the same meaning as before. For $k \in \polN_0$ and $u \in H^k(\Omega, \Lambda^\ell)$ we will set
\[
  \| u \|_{L^2(\Omega,\Lambda^\ell)}^2 = \sum_I \| u_I \|_{L^2(\Omega)}^2, \qquad | u |_{H^k(\Omega,\Lambda^\ell)}^2 = \sum_I | u_I |_{H^k(\Omega)}^2.
\]
We define, as usual, $H^k_0(\Omega,\Lambda^\ell)$ to be the closure of the space $C_0^\infty(\Omega,\Lambda^\ell)$ in the norm of $H^k(\Omega,\Lambda^\ell)$.

Let $\ell \in \{0, \ldots, n-1 \}$. For a smooth differential form $u \in C^1(\overline{\Omega}, \Lambda^\ell)$  we shall also need to define the trace $\tr_{\partial\Omega} u$. This can be done by invoking the inclusion $\mathfrak{i}: \partial\Omega \to \overline{\Omega}$ and its pullback
\[
  \tr_{\partial\Omega} u = \mathfrak{i}^\sharp u.
\]
An important feature of this mapping is that, if $\ell = n-1$, we have Stokes theorem \cite[Proposition 16.10]{MR1930091}:
\begin{equation}
\label{eq:StokesFormula}
  \int_{\partial\Omega} \tr_{\partial\Omega} u = \int_\Omega \diff u.
\end{equation}
This, in conjunction with Leibniz rule, yields that for every $u \in C^1(\overline{\Omega},\Lambda^\ell)$ and all $w \in C^1(\overline{\Omega}, \Lambda^{n-\ell-1})$
\[
  \int_\Omega \diff u \wedge w = (-1)^{\ell-1} \int_\Omega u \wedge \diff w + \int_{\partial\Omega} \tr_{\partial\Omega}u \wedge \tr_{\partial\Omega} w,
\]
which is sometimes called integration by parts.

As it is customary, see for example \cite[page 19]{MR2269741}, we extend this definition by continuity. In other words, every $u \in L^2(\Omega, \Lambda^\ell)$ with $\diff u \in L^2(\Omega, \Lambda^{\ell+1})$ defines a continuous linear functional, which we call $\tr_{\partial \Omega} u$, on $H^1(\Omega, \Lambda^{n-\ell-1})$ via
\[
  \langle \tr_{\partial\Omega} u, w \rangle = \int_\Omega \diff u \wedge w + (-1)^{\ell} \int_\Omega u \wedge \diff w.
\]

Finally, although this can be done more generally, we only define the interior product (contraction), denoted by $\ipd$, between a one--form and an $\ell$--differential form. Thus, if $z \in \Lambda^1 \approx \R^n$ and $u = \sum_I u_I \diff x_I \in L^1(\Omega,\Lambda^\ell)$, then
\[
  z \ipd u(x) = \sum_I u_I(x) \sum_{m=1}^\ell (-1)^{m-1} z_{i_m} \diff x_{\hat I_m} \in L^1(\Omega,\Lambda^{\ell-1}).
\]

\subsection{The Bogovski\u{\i} and regularized Poincar\'e integral operators}
\label{sub:Operators}

Let us now present the main objects that we are concerned with. From now on, we let $\theta \in C_0^\infty(\R^n)$ be supported on a ball $B$ and satisfy \eqref{eq:PropertiesTheta}. For $\ell \in \{0, \ldots, n\}$ we define the kernel $G_\ell$ by
\begin{equation}
\label{eq:Gell}
  G_\ell(x,y) = \int_1^\infty (t-1)^{n-\ell}t^{\ell-1}\theta( y + t(x-y) ) \diff t.
\end{equation}
The main objects of our concern in this work are the operators
\begin{align}
\label{eq:defBogovskii}
  \Bog u(x) &= \int G_\ell(x,y) (x-y) \ipd u(y) \diff y, \\
\label{eq:defPoincare}
  \Poinc u(x) &= \int G_{n-\ell+1}(y,x) (x-y) \ipd u(y) \diff y,
\end{align}
which we will call Bogovski\u{\i}--type and Poincar\'e--type operators, respectively. Here $\ell \in \{1, \ldots, n\}$ and, $u$ is an $\ell$--differential form. One of the main results, adapted to our needs, of \cite{MR2609313} is the following.

\begin{theorem}[continuity]
\label{thm:Costabel}
Let $\Omega$ be a bounded domain that is star shaped with respect to a ball containing $\supp \theta$, and $\ell \in \{1,\ldots,n\}$.
\begin{enumerate}[(1)]
  \item The operator $\Bog$, defined in \eqref{eq:defBogovskii}, defines a bounded linear operator on $L^2(\Omega, \Lambda^\ell)$. In addition, for $k \in \polN_0$, we have
  \[
    \| \Bog u \|_{H^{k+1}_0(\Omega,\Lambda^{\ell-1})} \leq C_{\Bog,k} \| u \|_{H^{k}_0(\Omega,\Lambda^\ell)}.
  \]
  Finally, if $u \in H^k_0(\Omega, \Lambda^\ell)$ is such that $\diff u = 0$, with $ \tr_{\partial\Omega} u = 0$ if $k=0$ and $\int_\Omega u  = 0$ if $\ell = n$, then
  \[
    u = \diff \Bog u.
  \]
  
  \item The operator $\Poinc$, defined in \eqref{eq:defPoincare}, defines a bounded linear operator on $L^2(\Omega, \Lambda^\ell)$. In addition, for $k \in \polN_0$, we have
  \[
    \| \Poinc u \|_{H^{k+1}(\Omega,\Lambda^{\ell-1})} \leq C_{\Poinc,k} \| u \|_{H^{k}(\Omega,\Lambda^\ell)}.
  \]
  Finally, if $u \in H^k(\Omega,\Lambda^\ell)$ is such that $\diff u = 0$, then
  \[
    u = \diff \Poinc u.
  \]
\end{enumerate}
\end{theorem}

Our main purpose in this work is to estimate the continuity constants, $C_{\Bog,1}$ and $C_{\Poinc,1}$, in this result.

\begin{remark}[continuity]
We must remark that a priori, the operators \eqref{eq:defBogovskii} and \eqref{eq:defPoincare} are weakly singular integral operators, and so care must be taken when manipulating them. However, one of the consequences of Theorem~\ref{thm:Costabel} is that these are bounded operators. For this reason, during the course of our estimates, we will change orders of integration and subdivide domains of integration with impunity.
\end{remark}

\section{The Poincar\'e--type operators}
\label{sec:Poincare}

Let us begin by providing an estimate on the continuity constant for the Poincar\'e--type operators. From now on, we will assume that our domain $\Omega$ is star shaped with respect to a ball $B$.

We begin by closely examining the operator. The change of variables $s = (t-1)/t$ allows us to rewrite the kernel $G_{n-\ell+1}$, defined in \eqref{eq:Gell}, as
\begin{align*}
  G_{n-\ell+1}(y,x) &= \int_1^\infty (t-1)^{n-(n-\ell+1)}t^{n-\ell+1-1} \theta(x + t(y-x)) \diff t \\
    &= \int_1^\infty (t-1)^{\ell-1} t^{n-\ell} \theta(x + t(y-x)) \diff t \\
    &= \int_0^1 \frac{ s^{\ell-1} }{(1-s)^{n+1} } \theta \left(x + \frac{y-x}{1-s} \right) \diff s.
\end{align*}
Therefore,
\[
  \Poinc u(x) = \int_0^1 s^{\ell-1} \int \theta \left(x + \frac{y-x}{1-s} \right) \frac{x-y}{1-s} \ipd u(y) \diff y \frac{\diff s}{(1-s)^n},
\]
so that, if $u(x) = \sum_I u_I(x) \diff x_I$, then
\[
  \Poinc u(x) =  \sum_{I} \sum_{m=1}^\ell (-1)^{m-1}
    \int_0^1 s^{\ell-1} \int \theta \left(x + \frac{y-x}{1-s} \right) \frac{x_m-y_m}{1-s} u_I(y) \diff y \frac{\diff s}{(1-s)^n}
   \diff x_{\hat I_m}.
\]

\subsection{First order estimates}
\label{sub:FOPoincare}

The computations presented above show that, to accomplish our goals, it suffices to consider, for $m \in \{1,\ldots, n\}$ and $f \in L^2(\R^n)$ such that $\supp f \subset \bar\Omega$, the operator
\begin{equation}
\label{eq:PoincareReduced}
  P_\ell f(x) = \int_0^1 s^{\ell-1} \int \theta \left(x + \frac{y-x}{1-s} \right) \frac{x_m-y_m}{1-s} f(y) \diff y \frac{\diff s}{(1-s)^n}.
\end{equation}
We, first of all, observe that $P_\ell f(x) = x_mP_1^\ell f(x) - P_2^\ell f(x)$ where, for $k \in \polN$, 
\begin{equation}
\label{eq:defPi}
  P_i^k f(x) = \int_0^1 s^{k-1} \int \phi_i\left(x + \frac{y-x}{1-s} \right) f(y) \diff y \frac{\diff s}{(1-s)^n},
\end{equation}
with
\[
  \phi_i(z) = 
  \begin{dcases}
    \theta(z), & i = 1, \\ \theta(z)z_m & i = 2.
  \end{dcases}
\]

As a consequence,
\begin{equation}
\label{eq:derivPoincare}
  \partial_j P_\ell f(x) = \delta_{j,m} P_1^\ell f(x) + x_m \partial_j P_1^\ell f(x) + \partial_j P_2^\ell f(x).
\end{equation}
Thus, for $k \in \polN$, we need to estimate two different types of operators:
\begin{align*}
  P_\partial^k f(x) &= \lim_{\eps \uparrow 1} \int_0^\eps \int s^{k-1} \partial_j \left[ \phi\left(x + \frac{y-x}{1-s} \right) \right] f(y) \diff y \frac{\diff s}{(1-s)^n}  , \\
  P_\theta^k f(x) &= \lim_{\eps \uparrow 1} \int_0^\eps \int s^{k-1} \theta \left(x + \frac{y-x}{1-s} \right)f(y) \diff y \frac{\diff s}{(1-s)^n} .
\end{align*}
We consider each one separately.

\subsubsection{Bound on $P_\partial^k$}
\label{subsub:boundPdk}
We split the integral that defines $P_\partial^k$ and set
\begin{align*}
    P_\partial^{k,L}f(x) &= \int_0^{1/2} \int s^{k-1} \partial_j \left[ \phi\left(x + \frac{y-x}{1-s} \right) \right] f(y) \diff y \frac{\diff s}{(1-s)^n}  , \\
    P_\partial^{k,U}f(x) &= \lim_{\eps \uparrow 1}   P_{\partial, \eps}^{k,U}f(x),
\end{align*}
where
\begin{equation*}
 P_{\partial, \eps}^{k,U}f(x) =\int_{1/2}^\eps \int s^{k-1} \partial_j \left[ \phi\left(x + \frac{y-x}{1-s} \right) \right] f(y) \diff y \frac{\diff s}{(1-s)^n}.
\end{equation*}

Let us now estimate $P_\partial^{k,U} f$. We will achieve this via the Fourier transform.
\begin{lemma}[Fourier transform]
\label{lem:FurtransPdU}
Let $f \in L^2(\R^n)$ with $\supp f \subset \bar \Omega$. The Fourier transform of $P_\partial^U f$ is
\[
  \widehat{P_\partial^{k,U} f}(\xi) = (-1)^n 2 \pi \imath \xi_j \int_{1/2}^1 s^{k-1} \hat f(\xi/s) \hat \phi ((s-1)\xi/s) \frac{\diff s}{s^n}.
\]
\end{lemma}
\begin{proof}
Taking the Fourier transform we get
%By definition, $P_\partial^{k,U} f(x) = \lim_{\eps \uparrow1} P_{\partial,\eps}^{k,U} f(x)$
\[
  \widehat{P_{\partial,\eps}^{k,U} f}(\xi) = \int_{1/2}^\eps s^{k-1} \int \partial_j \int \phi\left( x+ \frac{y-x}{1-s} \right) f(y) \diff y e^{-2\pi\imath \xi \cdot x} \diff x \frac{\diff s}{(1-s)^n}.
\]
Integration by parts shows that
\[
  \widehat{P_{\partial,\eps}^{k,U} f}(\xi) = 2\pi \imath \xi_j \int_{1/2}^\eps s^{k-1} \int f(y) \int \phi\left( x+ \frac{y-x}{1-s} \right) e^{-2\pi\imath \xi  \cdot x} \diff x \diff y \frac{\diff s}{(1-s)^n}.
\]
The change of variables $z = x+ \tfrac{y-x}{1-s} $ in the innermost integral shows that
\begin{align*}
  \widehat{P_{\partial,\eps}^{k,U} f}(\xi) &= 2\pi \imath \xi_j (-1)^n \int_{1/2}^\eps s^{k-1-n} \int f(y) e^{-2\pi \imath (\xi/s) \cdot  y} \diff y \int \phi(z) e^{-2\pi \imath ((s-1)\xi/s) \cdot z} \diff z \diff s \\
  &= (-1)^n 2 \pi \imath \xi_j \int_{1/2}^\eps s^{k-1} \hat f(\xi/s) \hat \phi ((s-1)\xi/s) \frac{\diff s}{s^n}.
\end{align*}
Letting $\eps \uparrow 1$ the result follows.
\end{proof}

The following result is similar to \cite[Lemma 2.3]{MR3086804}, but we provide a proof for completeness. To state it, and for future reference, we set the following notation. If $\rho>0$ and $\phi : \R^n \to \R$, then we define
\begin{equation}
  \mathsf{C}(\phi, \rho)=\rho^{-1} \| \phi \|_{L^1(\R^n)} + \rho \| \partial_j^2 \phi \|_{L^1(\R^n)}.
\end{equation}

\begin{lemma}[auxiliary estimate]
\label{lem:Lemma2.3Duran}
Let $\phi \in C_0^\infty(B)$, where $B$ is a ball of radius $\rho$, then
\[
  2\pi |\xi_j| \int_{-\infty}^0 |\hat \phi (t\xi) | \diff t \leq  \mathsf{C}(\phi, \rho).
\]
\end{lemma}
\begin{proof}
We write
\[
  2\pi |\xi_j| \int_{-\infty}^0 |\hat \phi (t\xi) | \diff t = 2\pi |\xi_j| \int_{-\frac1{2\pi \rho |\xi_j|} }^0 |\hat \phi (t\xi) | \diff t + 2\pi |\xi_j| \int_{-\infty}^{-\frac1{2\pi \rho |\xi_j|} } |\hat \phi (t\xi) | \diff t = I + II,
\]
and estimate each term separately. We have
\[
  I \leq 2 \pi |\xi_j| \| \hat \phi \|_{L^\infty} \int_{-\frac1{2\pi \rho |\xi_j|} }^0 \diff t = \rho^{-1} \| \hat \phi \|_{L^\infty(\R^n)} \leq \rho^{-1} \| \phi \|_{L^1(\R^n)},
\]
and
\[
  II = 2 \pi \|\xi_j^2 \hat \phi \|_{L^\infty(\R^n)} \int_{-\infty}^{-\frac1{2\pi \rho |\xi_j|} } \frac1{t^2|\xi_j|} \diff t \leq \rho \| \partial_j^2 \phi \|_{L^1(\R^n)},
\]
where, in both estimates, we used the Hausdorff--Young inequality \cite[Proposition 2.2.16]{MR3243734}. To conclude, collect both estimates.
\end{proof}

With these two results at hand we can finally bound $P_\partial^{k,U} f$.

\begin{proposition}[bound on $P_\partial^{k,U}$]
\label{prop:boundPpartialU}
Let $\Omega$ be star shaped with respect to a ball $B$ of radius $\rho$, $\phi \in C_0^\infty(B)$. Then, for every $f \in L^2(\R^n)$ with $\supp f \subset \bar\Omega$ we have that
\[
  \| P_\partial^{k,U} f \|_{L^2(\Omega)} \leq 2^{(n-2)/2} \mathsf{C}(\phi,\rho) \| f \|_{L^2(\Omega)}.
\]
% where
%\[
%  \mathsf{C}(\phi,\rho) = \rho^{-1} \| \phi \|_{L^1(\R^n)} + \rho \| \partial_j^2 \phi \|_{L^1(\R^n)}.
%\]
\end{proposition}
\begin{proof}
Applying the Cauchy--Schwarz inequality to $\widehat{P_\partial^{k,U} f}(\xi)$, which was obtained in Lemma~\ref{lem:FurtransPdU}, immediately yields that
\[
  |\widehat{P_\partial^{k,U} f}(\xi)|^2 \leq I \times II,
\]
\[
  I = 2\pi |\xi_j | \int_{1/2}^1 \left| \hat \phi \left( \frac{s-1}s \xi \right) \right| \frac{\diff s}{s^n}, \
  II = 2\pi |\xi_j| \int_{1/2}^1 s^{2(k-1)} \left| \hat \phi \left( \frac{s-1}s \xi \right) \right| |\hat f(\xi/s)|^2 \frac{\diff s}{s^n}.
\]
The change of variables $t=(s-1)/s$ and the Lemma~\ref{lem:Lemma2.3Duran} imply that
\[
  I = 2\pi |\xi_j | \int_{1/2}^1 \left| \hat \phi \left( \frac{s-1}s \xi \right) \right| \frac{\diff s}{s^n} =  2 \pi |\xi_j| \int_{-1}^0 (1-t)^{n-2} |\hat \phi(t\xi) |\diff t \leq 2^{n-2} \mathsf{C}(\phi,\rho).
\]

Integration in $\xi$ then reveals that
\begin{align*}
  \int |\widehat{P_\partial^{k,U} f}(\xi)|^2 \diff \xi &\leq 2^{n-2} \mathsf{C}(\phi,\rho) \int_{1/2}^1 s^{2(k-1)} \int 2\pi |\xi_j| \left| \hat \phi \left( \frac{s-1}s \xi \right) \right| |\hat f(\xi/s)|^2 \diff \xi \frac{\diff s}{s^n} \\
  &= 2^{n-2} \mathsf{C}(\phi,\rho) \int_{1/2}^1 s^{2k-1} \int 2\pi |z_j| \left| \hat \phi ( (s-1)z) \right| |\hat f(z)|^2 \diff z \diff s \\
  &= 2^{n-2} \mathsf{C}(\phi,\rho) \int |\hat f(z)|^2 \left( 2\pi |z_j| \int_{1/2}^1 s^{2k-1} \left| \hat \phi ( (s-1)z) \right| \diff s \right) \diff z \\
  &\leq 2^{n-2} \mathsf{C}(\phi,\rho)^2 \| \hat f \|_{L^2(\R^n)}^2,
\end{align*}
where we used Lemma~\ref{lem:Lemma2.3Duran} in the last step. Conclude using Plancherel's identity \cite[Theorem 2.2.14(4)]{MR3243734}.
\end{proof}

We will now estimate the term $P_\partial^{k,L}f(x)$, which we recall it reads
\begin{align*}
  P_\partial^{k,L} f(x) &= \int_0^{1/2} s^{k-1} \int \partial_j \left[ \phi \left( x + \frac{y-x}{1-s} \right) \right] f(y) \diff y \frac{\diff s}{(1-s)^n}  \\
    &=  -\int_0^{1/2} s^k \int \partial_j \phi \left( x + \frac{y-x}{1-s} \right) f(y) \diff y \frac{\diff s}{(1-s)^{n+1}}.
\end{align*}

The bound on this term depends on $k$. Thus, strategy that we will follow here is to obtain an $L^p(\Omega)$--estimate for suitable $p$ and interpolate it another bound that is valid for all values of $k$.

\begin{lemma}[$L^p(\Omega)$--estimate]
\label{lem:Lpestimate}
Let $k \in\polN$. If $p \in [1,\infty]$ is such that $p>n/(k+1)$, then
\[
  \| P_\partial^{k,L} f \|_{L^p(\Omega)} \leq \frac{2^{n/p-k}}{k-n/p+1} \| \partial_j \phi \|_{L^1(\R^n)} \| f \|_{L^p(\Omega)}.
\]
\end{lemma}
\begin{proof}
The change of variables $z = x + \tfrac{y-x}{1-s} $ shows that
\[
  |P_\partial^{k,L} f(x)| \leq \int_0^{1/2} \frac{s^k}{1-s} \int |\partial_j \phi(z)| | f(sx + (1-s) z)| \diff z \diff s,
\]
which by Minkowski's integral inequality \cite[Exercise 1.1.6]{MR3243734} implies that
\begin{align*}
  \| P_\partial^{k,L} f \|_{L^p(\Omega)} &\leq \int_0^{1/2} \frac{s^k}{1-s} \int |\partial_j \phi(z)| \left( \int_\Omega | f(sx + (1-s) z)|^p \diff x \right)^{1/p} \diff z \diff s \\
  &= \int_0^{1/2} \frac{s^{k-n/p}}{1-s} \int |\partial_j \phi(z)| \diff z \left( \int_\Omega | f(\bar x )|^p \diff \bar x \right)^{1/p} \diff s.
\end{align*}
Note that $x \in \Omega$ and $z \in \supp \phi \subset B$ so that, since $\Omega$ is star shaped with respect to a ball $\bar x = sx + (1-s) z \in \Omega$. The restriction on $p$ guarantees that the integrals converge and the result follows.
\end{proof}

The previous result implies an estimate for large enough $k$.

\begin{corollary}[estimate for large $k$]
\label{cor:BoundTLellbig}
Let $k \in \polN$ be such that $k > (n-2)/2$, then
\[
  \| P_\partial^{k,L} f \|_{L^2(\Omega)} \leq \frac{2^{n/2-k}}{k-n/2+1} \| \partial_j \phi \|_{L^1(\R^n)} \| f \|_{L^2(\Omega)}.
\]
\end{corollary}
\begin{proof}
The restriction on $k$ allows to apply the previous result with 
\[
  p = 2 > \frac{n}{k+1}.
\]
\end{proof}

If $k$ is not sufficiently large, we must proceed differently.

\begin{proposition}[estimate for small $k$]
\label{prop:estimasmallk}
Let $k \in \polN$ be such that $k \leq (n-2)/2$, then for any $p> n/(k+1)$ we have that
\[
  \| P_\partial^{k,L} f \|_{L^2(\Omega)} \leq \left( \frac{2^{n-k}}{k+1} \| \partial_j \phi \|_{L^\infty(\R^n)} |\Omega| \right)^\gamma \left( \frac{2^{n/p-k}}{k-n/p+1} \| \partial_j \phi \|_{L^1(\R^n)} \right)^{1-\gamma} \| f \|_{L^2(\Omega)},
\]
where $\gamma = \frac{p-2}{2(p-1)}$.
\end{proposition}
\begin{proof}
We begin by providing a bound in $L^1(\Omega)$. We have that
\[
  |P_\partial^{k,L} f(x) | \leq \int_0^{1/2} s^k \int \left|\partial_j \phi \left( x + \frac{y-x}{1-s} \right)\right| |f(y)| \diff y \frac{\diff s}{(1-s)^{n+1}} ,
\]
so that
\begin{align*}
  \| P_\partial^{k,L} f \|_{L^1(\Omega)} &\leq \| \partial_j \phi \|_{L^\infty(\R^n)} \| f \|_{L^1(\Omega)} |\Omega| \int_0^{1/2} s^k \frac{\diff s}{(1-s)^{n+1}} \\ &\leq 
    \frac{2^{n-k}}{k+1} \| \partial_j \phi \|_{L^\infty(\R^n)} |\Omega| \| f \|_{L^1(\Omega)} .
\end{align*}
Notice that
\[
  k + 1 \leq \frac{n}2 \quad \iff \quad \frac{n}{k+1} \geq 2.
\]
Thus, we can apply the Riesz--Thorin interpolation theorem \cite[Theorem 1.3.4]{MR3243734} between the recently obtained $L^1(\Omega)$ estimate and the $L^p(\Omega)$ estimate of Lemma~\ref{lem:Lpestimate} with $p > n/(k+1) \geq 2$ to obtain
\[
  \| P_\partial^{k,L} f \|_{L^2(\Omega)} \leq \left( \frac{2^{n-k}}{k+1} \| \partial_j \phi \|_{L^\infty(\R^n)} |\Omega| \right)^\gamma \left( \frac{2^{n/p-k}}{k-n/p+1} \| \partial_j \phi \|_{L^1(\R^n)} \right)^{1-\gamma} \| f \|_{L^2(\Omega)},
\]
where
\[
  \frac12 = \frac\gamma1 + \frac{1-\gamma}p \quad \implies \quad \gamma = \frac{p-2}{2(p-1)}.
\]
\end{proof}

We conclude by gathering all the previous estimates.

\begin{theorem}[bound on $P_\partial^k$]
\label{thm:boundPd}
Let $k \in \polN$. If $k > (n-2)/2$ then we have that
\[
  \| P_\partial^k f \|_{L^2(\Omega)} \leq \left[ 2^{(n-2)/2} \mathsf{C}(\phi,\rho) + \frac{2^{n/2-k}}{k-n/2+1} \| \partial_j \phi \|_{L^1(\R^n)} \right] \| f \|_{L^2(\Omega)}.
\]
If, on the other hand, $k \leq (n-2)/2$ then, for any $p>n/(k+1)$ we have that
\begin{multline*}
  \| P_\partial^k f\|_{L^2(\Omega)} \leq \left[2^{\frac{n-2}2} \mathsf{C}(\phi,\rho) + \right. \\ \left.\left( \frac{2^{n-k}}{k+1} \| \partial_j \phi \|_{L^\infty} |\Omega| \right)^\gamma \left( \frac{2^{n/p-k}}{k-n/p+1} \| \partial_j \phi \|_{L^1} \right)^{1-\gamma} \right]\| f\|_{L^2(\Omega)},
\end{multline*}
where
\[
  \gamma = \frac{p-2}{2(p-1)}.
\]
\end{theorem}
\begin{proof}
We have
\[
  \| P_\partial^k f \|_{L^2(\Omega)} \leq \| P_\partial^{k,L} f \|_{L^2(\Omega)} + \| P_\partial^{k,U} f \|_{L^2(\Omega)},
\]
and apply all the estimates that we have obtained so far.
\end{proof}

\subsubsection{Bound on $P_\theta^k$}
\label{subsub:boundPthetak}
We now bound the operator $P_\theta^k$. Once again, we define
\begin{align*}
    P_\theta^{k,L}f(x) &= \int_0^{1/2} \int s^{k-1} \theta \left(x + \frac{y-x}{1-s} \right)  f(y) \diff y \frac{\diff s}{(1-s)^n}  , \\
    P_\theta^{k,U}f(x) &= \lim_{\eps \uparrow 1} \int_{1/2}^\eps \int s^{k-1} \theta \left(x + \frac{y-x}{1-s} \right)  f(y) \diff y \frac{\diff s}{(1-s)^n} .
\end{align*}

The bound on $P_\theta^{k,U} f$ is immediate

\begin{lemma}[bound on $P_\theta^{k,U} f$]
\label{lem:boundPthetaU}
We have
\[
  \| P_\theta^{k,U} f \|_{L^2(\Omega)} \leq \frac{1 -2^{n/2-k}}{k-n/2} \| \theta \|_{L^1(\R^n)} \| f \|_{L^2(\Omega)}.
\]
\end{lemma}
\begin{proof}
%We can prove this result two different ways. %\AJSc{This is encouraging, as it shows we made the same mistake twice.}
%Via Fourier transform. 
Notice that
\[
  \widehat{P_\theta^{k,U}f}(\xi)= \int_{1/2}^1 \frac{s^{k-1}}{(1-s)^n} \int f(y) \int \theta \left(x + \frac{y-x}{1-s} \right) e^{-2\pi \imath \xi \cdot x} \diff x \diff y \diff s,
\]
and the change of variables $z = x + \tfrac{y-x}{1-s}$ in the innermost integral gives
\begin{align*}
  \widehat{P_\theta^{k,U} f}(\xi) &= (-1)^n \int_{1/2}^1 s^{k-n-1} \int f(y)e^{-2\pi \imath (\xi/s) \cdot y} \diff y \int \theta(z) e^{-2\pi \imath ((s-1)\xi/s) \cdot z } \diff z \diff s \\
   &= (-1)^n \int_{1/2}^1 s^{k-n-1} \hat f(\xi/s) \hat \theta( (s-1)\xi/s) \diff s.
\end{align*}
Minkowski's integral inequality then shows that
\begin{align*}
  \| \widehat{P_\theta^{k,U} f} \|_{L^2(\Omega)} &\leq \int_{1/2}^1 s^{k-n-1} \left( \int |\hat f(\xi/s)|^2 |\hat \theta ((s-1)\xi/s)|^2 \diff \xi \right)^{1/2} \diff s \\
    &\leq \| \hat \theta \|_{L^\infty(\R^n)} \int_{1/2}^1 s^{k-n/2-1} \left( \int |\hat f(\xi/s)|^2 \frac{\diff \xi}{s^n} \right)^{1/2} \diff s \\ &\leq \| \theta \|_{L^1(\R^n)} \| \hat f \|_{L^2(\R^n)} \int_{1/2}^1 s^{k-n/2-1} \diff s,
\end{align*}
where we again used Hausdorff--Young's inequality. Conclude using Plancherel's identity.
\end{proof}

We now bound the term $P_\theta^{k,L}f$, where again we distinguish two cases.

\begin{lemma}[$L^q$--estimate]
\label{lem:LqestPthetaL}
Let $k \in \polN$ and $q \in [1,\infty]$ be such that $q>n/k$, then 
\[
  \| P_\theta^{k,L} f \|_{L^q(\Omega)} \leq \frac{2^{n/q-k}}{k-n/q} \| \theta \|_{L^1} \| f \|_{L^q(\Omega)}.
\]
\end{lemma}
\begin{proof}
Similar to previous computations
\begin{align*}
  \| P_\theta^{k,L} f \|_{L^q(\Omega)} &\leq \int_0^{1/2} s^{k-1} \int \theta(z) \left( \int_\Omega |f(sx+(1-s)z)|^q \diff x \right)^{1/q} \diff z \diff s \\
  &= \int_0^{1/2} s^{k-1-n/q} \int \theta(z) \left( \int_\Omega |f(\bar x)|^q \diff\bar x \right)^{1/q} \diff z \diff s,
\end{align*}
and the condition on $q$ guarantees the convergence of the integrals.
\end{proof}

As a consequence, again, we obtain the desired $L^2(\Omega)$ estimate for large $k$.

\begin{corollary}[estimate for large $k$]
Let $k > n/2$, then
\[
  \| P_\theta^{k,L} f \|_{L^2(\Omega)} \leq \frac{2^{n/2-k}}{k-n/2} \| \theta \|_{L^1} \| f \|_{L^2(\Omega)}.
\]
\end{corollary}
\begin{proof}
Set $q = 2>n/k$ in Lemma~\ref{lem:LqestPthetaL}.
\end{proof}

For small values of $k$ we, again, proceed by interpolation.

\begin{proposition}[estimate for small $k$]
Let $k \in \polN$ be such that $k \leq n/2$. Then, for any $q>n/k$ we have that
\[
  \| P_\theta^{k,L} f \|_{L^2(\Omega)} \leq \left(\frac{2^{n-k}}k |\Omega| \| \theta \|_{L^\infty(\R^n)}\right)^\beta \left(\frac{2^{n/q-k}}{k-n/q} \| \theta \|_{L^1(\R^n)}\right)^{1-\beta} \| f \|_{L^2(\Omega)},
\]
where $\beta = \frac{q-2}{2(q-1)}$.
\end{proposition}
\begin{proof}
First, notice that,
\begin{align*}
  |P_\theta^{k,L}f(x)| &\leq \int_0^{1/2} \frac{s^{k-1}}{(1-s)^n} \int \theta \left(x + \frac{y-x}{1-s} \right) f(y) \diff y \diff s \\
    &\leq \| \theta \|_{L^\infty(\R^n)} \| f \|_{L^1(\Omega)} \int_0^{1/2} \frac{s^{k-1}}{(1-s)^n} \diff s,
\end{align*}
so that
\[
  \| P_\theta^{k,L}f \|_{L^1(\Omega)} \leq \frac{2^{n-k}}k |\Omega| \| \theta \|_{L^\infty(\R^n)} \| f \|_{L^1(\Omega)}.
\]

Let now $q > n/k \geq2$ and apply the Riesz--Thorin interpolation theorem to obtain
\[
  \| P_\theta^{k,L} f \|_{L^2(\Omega)} \leq \left(\frac{2^{n-k}}k |\Omega| \| \theta \|_{L^\infty(\R^n)}\right)^\beta \left(\frac{2^{n/q-k}}{k-n/q} \| \theta \|_{L^1(\R^n)}\right)^{1-\beta} \| f \|_{L^2(\Omega)},
\]
where
\[
  \frac12 = \frac\beta1 + \frac{1-\beta}q \quad \implies \quad \beta = \frac{q-2}{2(q-1)}.
\]
\end{proof}

We gather all estimates in one result.

\begin{theorem}[bound on $P_\theta^k$]
\label{thm:boundPtheta}
Let $k \in \polN$. If $k > n/2$, then we have that
\[
  \| P_\theta^k f \|_{L^2(\Omega)} \leq  \frac1{k-n/2} \| \theta \|_{L^1(\R^n)} \| f \|_{L^2(\Omega)}.
\]
If, on the other hand, $k \leq n/2$, then for any $q > n/k \geq 2$ we have that
\begin{multline*}
  \| P_\theta^kf \|_{L^2(\Omega) } \leq \left[\frac{1 -2^{\frac{n}2-k}}{k-n/2} \| \theta \|_{L^1(\R^n)} \right. \\ + \left. \left(\frac{2^{n-k}}k |\Omega| \| \theta \|_{L^\infty(\R^n)}\right)^\beta \left(\frac{2^{\frac{n}q-k}}{k-n/q} \| \theta \|_{L^1(\R^n)}\right)^{1-\beta} \right] \| f \|_{L^2(\Omega) }.
\end{multline*} 
with $\beta = \frac{q-2}{2(q-1)}$.
\end{theorem}
\begin{proof}
One only needs to gather all the estimates that we have obtained so far. The only point worth noting is that, in the case $k>n/2$, the constant in the estimate turns out to be
\[
  \frac{1 -2^{n/2-k}}{k-n/2} \| \theta \|_{L^1(\R^n)} + \frac{2^{n/2-k}}{k-n/2} \| \theta \|_{L^1(\R^n)}.
\]
\end{proof}

All these preparatory estimates allow us to obtain a bound on the operators $P^\ell_i$, for $i=1,2$, that comprise the components Poincar\'e--type operator $\Poinc$.

\begin{corollary}[bound on $P_1^\ell$]
\label{cor1}
If  $\ell \in \mathbb{N}$ and $ \frac{n}{2} < \ell \le n$ we have
\begin{equation}\label{cor11}
  \| P_1^\ell f \|_{L^2(\Omega)} \le  C(n, \ell) \|f\|_{L^2(\Omega)}. 
 \end{equation}
On the other hand, if  $\ell \in \mathbb{N}$ and $  1 \le \ell \le \frac{n}{2}$ we have
\begin{equation}\label{cor12}
   \| P_1^\ell f \|_{L^2(\Omega)} \le  C(n,\ell) 
    \left[ 1+ \left( \frac{|\Omega|}{|B|} \right)^{\frac{n-2\ell}{2(n-\ell)}} \left(\log\frac{|\Omega|}{|B|}\right)^\frac{n}{2(n-\ell)} \right].
\end{equation}
\end{corollary}
\begin{proof}
In the case  $\frac{n}{2} < \ell \le n$ we use the first case of Theorem~\ref{thm:boundPtheta}, to obtain
  \[
    \| P_1^\ell f \|_{L^2(\Omega)} \leq \frac1{\ell-n/2} \| \theta \|_{L^1(\R^n)} \| f \|_{L^2} = C(n,\ell) \| f \|_{L^2(\Omega)}.
  \]
  This shows \eqref{cor11}. In the case, $  1 \le \ell \le \frac{n}{2}$ we use  the second case of Theorem~\ref{thm:boundPtheta}  to obtain
 \begin{align*}
   \| P_1^\ell f \|_{L^2(\Omega)} &\leq   \left[ C(n,\ell) + \left(\frac{2^{n-\ell}}\ell |\Omega| \| \theta \|_{L^\infty(\R^n)}\right)^\beta \left(\frac{2^{n/q-\ell}}{\ell-n/q} \| \theta \|_{L^1(\R^n)}\right)^{1-\beta}\right] \| f\|_{L^2(\Omega)},
  \end{align*}
 with $q>n/\ell$ and $\beta = \frac{q-2}{2(q-1)}$.  Using that $\| \theta \|_{L^\infty(\R^n)} \leq C(n)/|B|$ we get 
\begin{align*}
   \| P_1^\ell f \|_{L^2(\Omega)} &\leq   D \| f\|_{L^2(\Omega)}, 
 \end{align*}
where
\begin{equation*}
 D=C(n, \ell) + \left( \frac{2^{n-\ell} C(n) }\ell \right)^\beta \left( \frac{2^{n/q-\ell}}{\ell-n/q} \right)^{1-\beta} \left( \frac{|\Omega|}{|B|} \right)^\beta.
\end{equation*}
  Let us write now
  \[
    \beta = \frac{q-2}{2(q-1)} = \frac12 \left( 1 - \frac1{q-1} \right) = \frac12 \left( 1 - \frac\ell{n-\ell} \right)+ \frac12 \left( \frac\ell{n-\ell} - \frac1{q-1} \right).
  \]
  We choose $q = n/\ell + \epsilon$  with 
  \begin{equation*}
  \epsilon=  \frac{\mathcal{C} A^2}{1-\mathcal{C} A},  \qquad \mathcal{C}= \frac2{\log(|\Omega|/|B|)}   \qquad  A={n / \ell-1}. 
  \end{equation*}
 We assume that $|\Omega|/|B|$ is sufficiently large so that $1-\mathcal{C} A >1/2$ which will imply that  $0< \epsilon  \le  A$. We then see that
  \[
    \frac12 \left( \frac\ell{n-\ell} - \frac1{q-1} \right) = \frac1{\log(|\Omega|/|B|)}.
  \]
  Consequently,
  \[
 D \leq C(n, \ell) + e\left( \frac{2^{n-\ell} C(n) }\ell \right)^\beta \left( \frac{2^{n/q-\ell}}{\ell-n/q} \right)^{1-\beta} \left( \frac{|\Omega|}{|B|} \right)^{\frac{n-2\ell}{2(n-\ell)}}.
  \]
  
With these choices, we get that
  \begin{equation*}
  \ell -n/q = \frac{\ell^2 \epsilon}{n+ \ell \epsilon} \ge  \frac{\ell^2 \epsilon}{n+ \ell A} \ge   \mathcal{C} \frac{\ell^2 A^2 }{n+ \ell A}.
  \end{equation*}
  Moreover, 
  \begin{equation*}
  1-\beta=q/2(q-1)= \frac{n+ \epsilon \ell}{2 (n-\ell)+ 2 \epsilon \ell} \le   \frac{n}{2 (n-\ell)}.
  \end{equation*}

  Hence,
  \[
  D \leq C(n,\ell) 
    \left[ 1+ \left( \frac{|\Omega|}{|B|} \right)^{\frac{n-2\ell}{2(n-\ell)}} \left(\log\frac{|\Omega|}{|B|}\right)^\frac{n}{2(n-\ell)} \right].
  \]
  
 This proves \eqref{cor12}.
\end{proof}

\begin{corollary}[bound on $\partial_j P_i^\ell$]
\label{cor2}
If  $\ell \in \mathbb{N}$ and $ \frac{n-2}{2} < \ell \le n$ we have
\begin{equation}\label{cor21}
 R \| \partial_j P_1^\ell f \|_{L^2(\Omega)} + \| \partial_j P_2^\ell f \|_{L^2(\Omega)} \le  C(n, \ell) \frac{R}{\rho} \|f\|_{L^2(\Omega)}. 
 \end{equation}
On the other hand, if  $\ell \in \mathbb{N}$ and $  1 \le \ell  \le \frac{n-2}{2}$ we have
\begin{multline}
\label{cor22}
 R \| \partial_j P_1^\ell f \|_{L^2(\Omega)} + \| \partial_j P_2^\ell f \|_{L^2(\Omega)} \le  \\ C(n,\ell) \frac{R}{\rho}
    \left[ 1+ \left( \frac{|\Omega|}{|B|} \right)^{\frac{n-2(\ell+1)}{2(n-\ell-1)}}  \left( \log\frac{|\Omega|}{|B|} \right)^{\frac{n}{2(n-\ell-1)}} \right] \| f \|_{L^2(\Omega)}.
\end{multline}
\end{corollary}
\begin{proof}
If  $ \frac{n-2}{2} < \ell \le n$ we use the first case of Theorem~\ref{thm:boundPd} to get
  \begin{align*}
    \| \partial_j P_1^\ell f \|_{L^2(\Omega)} &\leq \left[ 2^{(n-2)/2} \mathsf{C}(\phi_1,\rho) + \frac{2^{n/2-\ell}}{\ell-n/2+1} \| \partial_j \phi_1 \|_{L^1(\R^n)} \right] \| f \|_{L^2(\Omega)} \\
      &\leq C(n, \ell) \left[ \mathsf{C}(\phi_1,\rho) + \frac1\rho \right] \| f \|_{L^2(\Omega)},
  \end{align*}
  where we used that $\phi_1 = \theta$ and \eqref{eq:PropertiesTheta}. Using \eqref{eq:PropertiesTheta} again yields
  \begin{equation}
  \label{eq:boundCphi1}
    \mathsf{C}(\phi_1,\rho) = \mathsf{C}(\theta,\rho) = \rho^{-1} \| \theta \|_{L^1(\R^n)} + \rho \| D^2 \theta \|_{L^1(\R^n)} \leq \frac{C(n)}\rho,
  \end{equation}
  and as a consequence we finally get
  \[
    R \| \partial_j P_1^\ell f \|_{L^2(\Omega)} \leq \frac{C(n,\ell)R}\rho \| f \|_{L^2(\Omega)}.
  \]

  Similarly,
  \begin{align*}
    \| \partial_j P_2^\ell f \|_{L^2(\Omega)} &\leq \left[ 2^{(n-2)/2} \mathsf{C}(\phi_2,\rho) + \frac{2^{n/2-\ell}}{\ell-n/2+1} \| \partial_j \phi_2 \|_{L^1(\R^n)} \right] \| f \|_{L^2(\Omega)} \\
      &\leq C(n, \ell) \left[ \mathsf{C}(\phi_2,\rho) + \| \partial_j \phi_2 \|_{L^1(\R^n)} \right] \| f \|_{L^2(\Omega)}.
  \end{align*}
 Using that $\phi_2(z) = z_k \theta(z)$, we can easily show that  $\mathsf{C}(\phi_2,\rho) + \| \partial_j \phi_2 \|_{L^1(\R^n)} \le C(n)$. Consequently, 
 \begin{align*}
    \| \partial_j P_2^\ell f \|_{L^2(\Omega)} \le   C(n, \ell) \| f \|_{L^2(\Omega)}.
 \end{align*}
 This proves \eqref{cor21}.
 
 If  $  1 \le \ell  \le \frac{n-2}{2}$  then by Theorem \ref{thm:boundPd} we have
  \[
    \| \partial_j P_1^\ell f \|_{L^2(\Omega)} \leq D \| f \|_{L^2(\Omega)},
  \]
  with
  \[
    D=  2^{(n-2)/2} \mathsf{C}(\phi_1,\rho) + \left( \frac{2^{n-\ell}}{\ell+1} \| \partial_j \phi_1 \|_{L^\infty(\R^n)} |\Omega| \right)^\gamma \left( \frac{2^{n/p-\ell}}{\ell-n/p+1} \| \partial_j \phi_1 \|_{L^1(\R^n)} \right)^{1-\gamma}  .
  \]
Here  $p>n/(\ell+1)$, and $\gamma = \tfrac{p-2}{2(p-1)}$. Using \eqref{eq:boundCphi1} and \eqref{eq:PropertiesTheta} this reduces to
  \[
    D \leq \frac{C(n,\ell)}\rho \left[
      1 + \left( \frac{2^{n-\ell} }{\ell+1} \right)^\gamma \left( \frac{2^{n/p-\ell}}{\ell-n/p+1} \right)^{1-\gamma}
      \left(\frac{|\Omega|}{|B|}\right)^\gamma \right].
  \]
  We write
  \[
    \gamma = \frac12 \left( 1 - \frac1{p-1} \right) = \frac12\left( 1- \frac{\ell+1}{n-(\ell+1)} \right) + \frac12\left( \frac{\ell+1}{n-(\ell+1)} - \frac1{p-1} \right).
  \]
We choose $p = n/(\ell+1) + \epsilon$  with 
\begin{equation*}
\epsilon=  \frac{\mathcal{C} A^2}{1-\mathcal{C} A},  \qquad \mathcal{C}= \frac2{\log(|\Omega|/|B|)},   \qquad A=\frac{1}{n / (\ell+1)-1}. 
\end{equation*}
 We assume that $|\Omega|/|B|$ is sufficiently large so that $1-\mathcal{C} A >1/2$ which will imply that  $0< \epsilon  \le  A$. We then see that
  \[
    \frac12 \left( \frac{\ell+1}{n-(\ell+1)} - \frac1{p-1} \right) = \frac1{\log(|\Omega|/|B|)}.
  \]
  Consequently, 
  \[
    D \leq \frac{C(n,\ell)}\rho \left[
      1 + e \left( \frac{2^{n-\ell} }{\ell+1} \right)^\gamma \left( \frac{2^{n/p-\ell}}{\ell-n/p+1} \right)^{1-\gamma}
      \left(\frac{|\Omega|}{|B|}\right)^{ {\frac{n-2(\ell+1)}{2(n-\ell-1)}}} \right].
  \]
  
 We see that these choices yield
  \begin{equation*}
  \ell+1 -n/p = \frac{(\ell+1)^2 \epsilon}{n+ (\ell+1) \epsilon} \ge  \frac{(\ell+1)^2 \epsilon}{n+ (\ell+1) A} \ge   \mathcal{C} \frac{(\ell+1)^2 A^2 }{n+ (\ell+1) A}.
  \end{equation*}
 Moreover, 
  \begin{equation*}
  1-\gamma=(p-1)/2p= \frac{n+ \epsilon (\ell+1)}{2 (n-(\ell+1))+ 2 \epsilon (\ell+1)} \le   \frac{n}{2 (n-(\ell+1))}.
  \end{equation*}
  As a consequence,
  \begin{equation*}
  \left( \frac{2^{n/p-\ell}}{\ell-n/p+1} \right)^{1-\gamma} \le  C(n,\ell)  \left(\log\frac{|\Omega|}{|B|} \right)^{\frac{n}{2(n-\ell-1)}}.
  \end{equation*}
  Thus,
  \begin{equation*}
  D\le   \frac{C(n,\ell)}{\rho}
  \left[ 1+ \left( \frac{|\Omega|}{|B|} \right)^{\frac{n-2(\ell+1)}{2(n-\ell-1)}}  \left( \log\frac{|\Omega|}{|B|} \right)^{\frac{n}{2(n-\ell-1)}} \right].
    \end{equation*} 
    which shows that 
   \begin{equation*}
   R \| \partial_j P_1^\ell f \|_{L^2(\Omega)}  \le   C(n, \ell) \frac{R}{\rho}
  \left[ 1+ \left( \frac{|\Omega|}{|B|} \right)^{\frac{n-2(\ell+1)}{2(n-\ell-1)}}  \left( \log\frac{|\Omega|}{|B|} \right)^{\frac{n}{2(n-\ell-1)}} \right].
    \end{equation*} 
     Similarly, we can show that 
    \begin{equation*}
    \| \partial_j P_2^\ell f \|_{L^2(\Omega)}  \le   C(n, \ell) 
  \left[ 1+ \left( \frac{|\Omega|}{|B|} \right)^{\frac{n-2(\ell+1)}{2(n-\ell-1)}}  \left( \log\frac{|\Omega|}{|B|} \right)^{\frac{n}{2(n-\ell-1)}} \right].
    \end{equation*} 
       This proves \eqref{cor22}.
\end{proof}

\subsubsection{The final estimate}
\label{sub:finalestP}

All the preliminary estimates of the previous section make obtaining a first order estimate for $\Poinc$ almost immediate.

%All the previous estimates can now be applied to \eqref{eq:derivPoincare} to obtain a bound on derivatives of the Poincar\'e--type operators defined in \eqref{eq:defPoincare}. Notice however, that the results of Theorems~\ref{thm:boundPd} and \ref{thm:boundPtheta} have two cases each. Thus, we have four cases:
%
%\begin{center}
%  \begin{tabular}{||l|c|c||}
%    \hline
%    & $P_\partial^\ell$ & $P_\theta^\ell$\\
%    \hline
%    & $\ell > (n-2)/2$ & $\ell > n/2$ \\
%    \hline
%    Case 1: & True & True \\
%    Case 2: & False & True \\
%    Case 3: & True & False \\
%    Case 4: & False & False \\
%    \hline
%  \end{tabular}  
%\end{center}
%
%Let us see what each case means separately:
%\begin{description}
%  \item[Case 1:] Here we have $2\ell >n$.
%  
%  \item[Case 2:] In this situation we must have that
%  \[
%    n < 2\ell \leq n-2,
%  \]
%  which is a contradiction.
%  
%  \item[Case 3:] Here we have
%  \[
%    n-2<  2 \ell \leq n,
%  \]
%  so that either $2\ell = n-1$ or $2\ell = n$. Notice that only one of these two choices is possible.
%  
%  \item[Case 4:] Finally, $2\ell \leq n-2$.
%\end{description}

\begin{theorem}[bound on $P_\ell$]
\label{thm:boudnP}
Let $\Omega$ be a bounded domain that is star shaped with respect to a ball $B$. Set $R=\diam(\Omega)$, and $\rho = \diam(B)$. Then, for $\ell \in \{1, \ldots, n \}$, the operator $P_\ell$, defined in \eqref{eq:PoincareReduced}, satisfies
\[
  | P_\ell f |_{H^1(\Omega)} \leq C(n, \ell)  \frac{R}\rho \kappa \| f \|_{L^2(\Omega)},
\]
where $C(n,\ell)$ is a constant that only depends on $n$ and $\ell$ and $\kappa = \kappa( \Omega, B, R, \rho)$ is such that,
\begin{enumerate}[1.]
  \item If $2\ell >n$, then
  \[
    \kappa = 1.
  \]
  
  \item If $2\ell =n-1$ or $2\ell = n$, then
  \[
    \kappa =\left( \frac{|\Omega|}{|B|} \right)^{\frac{n-2\ell}{2(n-\ell)}} \left(\log\frac{|\Omega|}{|B|}\right)^\frac{n}{2(n-\ell)}.
  \]
  
  \item If $2\ell \leq n-2$, then
  \[
    \kappa =
      \left( \frac{|\Omega|}{|B|} \right)^{\frac{n-2\ell}{2(n-\ell)}}  \left( \log\frac{|\Omega|}{|B|} \right)^{\frac{n}{2(n-\ell-1)}} .
  \]
\end{enumerate}
\end{theorem}
\begin{proof}
From \eqref{eq:derivPoincare} we have that, for any $j \in \{1, \ldots, n\}$,
\[
  \| \partial_j P_\ell f\|_{L^2(\Omega)} \leq \| P_1^\ell f \|_{L^2(\Omega)} + R \| \partial_j P_1^\ell f \|_{L^2(\Omega)} + \| \partial_j P_2^\ell f \|_{L^2(\Omega)}.
\]
Then, the result follows from Corollaries \ref{cor1} and \ref{cor2}. 
\end{proof}

The most important consequence of this estimate is one of the main results in this work. Namely, an estimate on $C_{\Poinc,1}$.

\begin{corollary}[estimate on $C_{\Poinc,1}$]
\label{cor:PoincareEstimate}
Let $\Omega$ be a bounded domain that is star shaped with respect to a ball $B$. Set $R=\diam(\Omega)$, and $\rho = \diam(B)$. Then, for $\ell \in \{1, \ldots, n \}$, the operator $\Poinc$, defined in \eqref{eq:defPoincare} satisfies
\[
  | \Poinc u |_{H^1(\Omega,\Lambda^{\ell-1})} \leq C(n, \ell)  \frac{R}\rho \kappa \| u \|_{L^2(\Omega, \Lambda^\ell)},
\]
where $C(n,\ell)$ is a constant that only depends on $n$ and $\ell$ and $\kappa = \kappa( \Omega, B, R, \rho)$ is as in Theorem~\ref{thm:boudnP}.
\end{corollary}
\begin{proof}
It suffices to apply the estimate of Theorem~\ref{thm:boudnP} to each one of the components of $\Poinc u$.
\end{proof}

\subsection{Second order estimates}
\label{sub:HkPoincare}

Let us now describe how our techniques can be used to estimate the continuity constant in the case that $u \in H^k(\Omega,\Lambda^\ell)$ with $k \in \polN$. The starting point is again the operator $P_\ell$ defined in \eqref{eq:PoincareReduced}. The change of variables $z = x+(y-x)/(1-s)$ reveals (compare with formula (3.9) of \cite{MR2609313})
\[
  P_\ell f(x) = \int \theta(z) (x_m - z_m) \int_0^1 s^{\ell-1} f(sx + (1-s)z) \diff s\diff z,
\]
and, therefore, if $\alpha \in \polN^n$ is a multiindex of length $k$
\begin{align*}
  \partial^\alpha P_\ell f(x) &=  \sum_{\nu \leq \alpha} {\alpha \choose \nu} \partial^\nu(x_m) \int \theta(z) \int_0^1 s^{\ell-1} \partial_x^{\alpha - \nu}[f(sx + (1-s)z)] \diff s \diff z \\
    &- \int z_m \theta(z) \int_0^1 s^{\ell-1} \partial_x^\alpha[f(sx + (1-s)z)] \diff s \diff z \\
    &=  \sum_{\nu \leq \alpha} {\alpha \choose \nu} \partial^\nu(x_m) \int \theta(z) \int_0^1 s^{\ell-1 + |\alpha-\nu|} \partial^{\alpha - \nu}f(sx + (1-s)z) \diff s \diff z \\
    &- \int z_m \theta(z) \int_0^1 s^{\ell-1+k} \partial^\alpha f(sx + (1-s)z) \diff s \diff z,
\end{align*}
where the sum is over all multiindices $\nu \in \polN^n$ such that $\nu_i \leq \alpha_i$ for all $i \in \{1, \ldots, n\}$ and
\[
  {\alpha \choose \nu } = \prod_{i=1}^n {\alpha_i \choose \nu_i}.
\]
Let us now introduce the change of variables $y = sx + (1-s)z$ in each integral to obtain
\begin{align*}
  \partial^\alpha P_\ell f(x) &= 
  \sum_{\nu \leq \alpha} {\alpha \choose \nu} \partial^\nu(x_m) \int_0^1 s^{\ell-1 + |\alpha-\nu|} \int \theta\left( x + \frac{y-x}{1-s} \right) \partial^{\alpha - \nu} f(y) \diff y \frac{\diff s}{(1-s)^n} \\
  & - \int_0^1 s^{\ell-1+k} \int \theta\left( x + \frac{y-x}{1-s} \right) \frac{x_m - y_m}{1-s} \partial^{\alpha } f(y) \diff y \frac{\diff s}{(1-s)^n} \\
  &= \sum_{\nu \leq \alpha} {\alpha \choose \nu} \partial^\nu(x_m) P_1^{\ell + |\alpha - \nu|} [\partial^{\alpha-\nu}f](x) + 
    P_2^{\ell + k}[\partial^\alpha f](x),
\end{align*}
where the operators $P_i^k$, for $i = 1,2$ and $k \in \polN$ are defined in \eqref{eq:defPi}. Thus, in much similarity to \eqref{eq:derivPoincare}, we have that
\begin{equation}
\label{eq:HighDerivsPoincare}
\begin{aligned}
  \partial_j \partial^\alpha P_\ell f(x) &= 
  \delta_{jm} P_1^{\ell + k} [\partial^{\alpha}f](x) +
  \sum_{0<\nu \leq \alpha} {\alpha \choose \nu} \partial^\nu(x_m) \partial_j P_1^{\ell + |\alpha - \nu|} [\partial^{\alpha-\nu}f](x) \\ &+ 
   \partial_j P_2^{\ell + k}[\partial^\alpha f](x).
\end{aligned}
\end{equation}

The bounds on each one of these operators were already obtained in Sections~\ref{subsub:boundPdk} and \ref{subsub:boundPthetak}. Unfortunately, since they heavily depend on the order of the operator, there is no clear way that one can explicitly estimate the middle term in this last expression. For this reason, we will content ourselves with a second order estimate in the case of sufficiently large $\ell$. The remaining cases can be treated with a detailed analysis similar to that of Section~\ref{sub:finalestP}. We skip this for brevity.

\begin{theorem}[estimate on $C_{\Poinc,2}$]
\label{thm:H2estPoincare}
Let $\Omega$ be a bounded domain that is star shaped with respect to a ball $B$. Set $R=\diam(\Omega)$, and $\rho = \diam(B)$. Then, for $\ell \in \{1, \ldots, n \}$, such that $2\ell >n$ the operator $\Poinc$, defined in \eqref{eq:defPoincare}, satisfies
\[
  | \Poinc u |_{H^2(\Omega,\Lambda^{\ell-1})} \leq C(n, \ell)  \frac1\rho \| u \|_{H^1(\Omega, \Lambda^\ell)},
\]
where $C(n,\ell)$ is a constant that only depends on $n$ and $\ell$.
\end{theorem}
\begin{proof}
Clearly, it suffices to estimate each one of the terms in \eqref{eq:HighDerivsPoincare} for $k=1$ and multiindices $\nu$ such that $|\alpha - \nu| = 0$. Appealing to Theorems~\ref{thm:boundPd} and \ref{thm:boundPtheta} with $k = \ell+1$ then we get that
\[
  | \Poinc u |_{H^2(\Omega, \Lambda^{\ell-1})} \leq (C_1 + C_2 + C_3) \| u \|_{H^1(\Omega, \Lambda^\ell)},
\]
with
\[
  C_1 = \frac1{\ell+1-n/2},
\]
and, for $i=2,3$,
\[
  C_i = C_i(\phi) = 2^{(n-2)/2} \mathsf{C}(\phi,\rho) + \frac{2^{n/2-\ell-1}}{\ell-n/2+2} \| \partial_j \phi \|_{L^1(\R^n)},
\]
where $\phi(z) = z_m \theta(z)$ for $i=2$ and $\phi(z) = \theta(z)$ for $i=3$. Estimates \eqref{eq:boundCphi1} and properties of $\theta$ show that
\[
  C_2 \leq C(n), \qquad C_3 \leq \frac{C(n)}\rho.
\]
This is the claimed estimate.
\end{proof}

\section{The Bogovski\u{\i}--type operators}
\label{sec:Bogovskii}

In this section we obtain bounds on the Bogovski\u{\i}--type operator defined in \eqref{eq:defBogovskii}. To keep the presentation within reasonable limits, many of the computations will be skipped as they repeat much of what we have already accomplished for the Poincar\'e--type operator in previous sections.

A simple change of variables allows us to write
\[
  \Bog u(x) = \int_0^1 (1-s)^{n-\ell} \int \theta \left(y + \frac{x-y}{s} \right) \frac{x-y}{s} \ipd u(y) \diff y \frac{\diff s}{s^n},
\]
so that, if $u(x) = \sum_I u_I(x) \diff x_I$, then
\[
  \Bog u(x) =  \sum_{I} \sum_{m=1}^\ell (-1)^{m-1}
    \int_0^1 (1-s)^{n-\ell} \int \theta  \left(y + \frac{x-y}{s} \right) \frac{x_m-y_m}{s} u_I(y) \diff y \frac{\diff s}{s^n}
   \diff x_{\hat I_m}.
\]

The computations presented above show that, to accomplish our goals, it suffices to consider, for $m \in \{1,\ldots, n\}$ and $f \in L^2(\R^n)$ such that $\supp f \subset \bar\Omega$ the operator
\begin{equation}
\label{eq:BogReduced}
  Q_\ell f(x) = \int_0^1  (1-s)^{n-\ell} \int  \theta \left(y + \frac{x-y}{s} \right) \frac{x_m-y_m}{s} f(y) \diff y \frac{\diff s}{s^n}.
\end{equation}
We, first of all, observe that $Q_\ell f(x) = -Q_1^\ell g(x) + Q_2^\ell f(x)$ where, $g(y)=yf(y)$ and  for $k \in \polN$, 
\begin{equation}
\label{eq:defQi}
  Q_i^\ell v(x) =  \int_0^1  (1-s)^{n-\ell} \int  \phi_i \left(y + \frac{x-y}{s} \right) v(y) \diff y \frac{\diff s}{s^n},
\end{equation}
with
\[
  \phi_i(z) = 
  \begin{dcases}
    \theta(z), & i = 1, \\ \theta(z)z_m & i = 2.
  \end{dcases}
\]
As a consequence,
\begin{equation}
\label{eq:derivBogovskii}
  \partial_j Q_\ell f(x) = -\partial_j Q_1^\ell g(x) + \partial_j Q_2^\ell f(x),
\end{equation}
where $g(y)=yf(y)$. Thus, for $k \in \polN$ we need to estimate the following type of operator:
\begin{align*}
  Q_\partial^\ell v(x) &= \lim_{\eps \downarrow 0} \int_\eps^1  (1-s)^{n-\ell} \int  \partial_j\left[\phi \left(y + \frac{x-y}{s} \right)\right] v(y) \diff y \frac{\diff s}{s^n}. \\
\end{align*}
We write $Q_\partial^\ell v(x) =  Q_\partial^{\ell,L}v(x)+ Q_\partial^{\ell,U}v(x) $ where
\begin{align*}
    Q_\partial^{\ell,L}v(x) &= \lim_{\eps \downarrow 0}   Q_{\partial,\eps}^{\ell,L}v(x)  , \\
    Q_\partial^{\ell,U}v(x) &=   \int_{1/2}^1  (1-s)^{n-\ell} \int  \partial_j\left[\phi \left(y + \frac{x-y}{s} \right)\right] v(y) \diff y \frac{\diff s}{s^n},
\end{align*}
and
\begin{equation*}
 Q_{\partial, \eps}^{\ell,L}v(x) =\int_{\eps}^{1/2}  (1-s)^{n-\ell} \int  \partial_j\left[\phi \left(y + \frac{x-y}{s} \right)\right] v(y) \diff y \frac{\diff s}{s^n}.
\end{equation*}

As in the case of the Poincar\'e operator, we estimate each one of these separately. It is interesting to note that the techniques used here are, in a sense, dual to those needed in previous section.

\subsection{Bound for  $Q_\partial^{\ell,L}$}

We begin by bounding $Q_\partial^{\ell,L}$. This will be accomplished via the Fourier transform.

\begin{lemma}[Fourier transform]
We have that
\begin{equation}\label{aux1}
  \widehat{ Q_\partial^{\ell,L}v}(\xi)=2 \pi \imath \xi_j \int_{0}^{\frac{1}{2}}  (s-1)^{n-\ell} \widehat{\phi}(s \xi) \widehat{v}((1-s) \xi) \diff s. 
\end{equation}
\end{lemma}
\begin{proof}
We take the Fourier transform:
\begin{equation}\label{aux2}
  \widehat{ Q_{\partial,\eps}^{\ell,L}v}(\xi)= \int \int \int_{\eps}^{\frac{1}{2}}  (s-1)^{n-\ell} \partial_{j} \left[\phi\left( y+\frac{x-y}{s} \right)  \right] v(y) e^{-2 \pi \imath x \cdot \xi}  \frac{ds}{s^n} \diff y \diff x.
\end{equation}
We integrate by parts to obtain
\begin{equation*}
  \widehat{ Q_{\partial,\eps}^{\ell,L} v }(\xi)= 2 \pi \imath \xi_j  \int_{\eps}^{\frac{1}{2}}  (s-1)^{n-\ell} \int \int  \phi\left( y+\frac{x-y}{s} \right)   v(y) e^{-2 \pi \imath x \cdot \xi}  \diff x \diff y  \frac{\diff s}{s^n}.
\end{equation*}
Making the change of variables  $z=y + \tfrac{x-y}{s}$ we get 
\begin{equation*}
  \widehat{Q_{\partial,\eps}^{\ell,L}v}(\xi)= 2 \pi \imath \xi_j  \int_{\eps}^{\frac{1}{2}}  (s-1)^{n-\ell} \int \int  \phi( z)   v(y) e^{-2 \pi \imath (sz+(1-s) y) \cdot \xi}  \diff z \diff y \diff s.
\end{equation*}
Thus, 
\begin{alignat*}{1}
  \widehat{Q_{\partial,\eps}^{\ell,L}v}(\xi)=& 2 \pi \imath \xi_j  \int_{\eps}^{\frac{1}{2}}  (s-1)^{n-\ell}  \hat{\phi}( s\xi) \int    v(y) e^{-2 \pi \imath (1-s) y \cdot \xi}  \diff y \diff s\\
  =& 2 \pi \imath \xi_j  \int_{\eps}^{\frac{1}{2}}  (s-1)^{n-\ell}  \hat{\phi}( s\xi)   \hat{v}((1-s) \xi)    \diff s.
\end{alignat*}
The identity \eqref{aux1} follows by taking the limit $\eps \downarrow 0$.
\end{proof}

\begin{lemma}[estimate on $Q_\partial^{\ell,L}$]
It holds that
\begin{equation}\label{QL}
  \| Q_{\partial}^{\ell,L}v\|_{L^2(\Omega)} \le   c_{n+1-2\ell} \mathsf{C}(\phi,\rho ) \|v\|_{L^2(\Omega)},
\end{equation}
where $c_r^2= \max_{0 \le t \le 1} \frac{1}{|1+t|^r}$.
\end{lemma}
\begin{proof}
By using  \eqref{aux1}, the Cauchy--Schwarz inequality, and Lemma~\ref{lem:Lemma2.3Duran} we get
\begin{alignat*}{1}
  |\widehat{Q_\partial^{\ell,L}v}(\xi)|^2 \le &  2\pi \mathsf{C}(\phi,\rho )   |\xi_j| \int_{0}^{\frac{1}{2}}  |1-s|^{2(n-\ell)} |\widehat{\phi}(s \xi)| |\widehat{v}((1-s) \xi)|^2 \diff s . 
\end{alignat*}
Hence,  we obtain 
\begin{alignat*}{1}
  \left\|\widehat{Q_\partial^{\ell,L}v} \right\|_{L^2(\mathbb{R}^n)}^2  \le  2\pi \mathsf{C}(\phi,\rho )  \int |\xi_j| \int_{0}^{\frac{1}{2}}  |1-s|^{2(n-\ell)}    |\widehat{\phi}(s \xi)| |\widehat{v}((1-s) \xi)|^2 \diff s   \diff \xi \Big.
\end{alignat*}
Using the change of variables $\eta =(1-s) \xi$, we see that  
\begin{alignat*}{1}
  \left\| \widehat{Q_\partial^{\ell,L}v} \right\|_{L^2(\mathbb{R}^n)}^2  \le    2\pi \mathsf{C}(\phi,\rho )   \int  \int_{0}^{\frac{1}{2}}   |1-s|^{2(n-\ell)-n-1} |\eta_j| \left|\widehat{\phi}\left(\frac{s \eta}{1-s} \right) \right| \left|\widehat{v}(\eta)\right|^2 \diff s \diff \eta.
\end{alignat*}
Another change of variables  $t= \frac{s}{1-s}$ gives
\begin{alignat*}{1}
  \left\| \widehat{Q_\partial^{\ell,L}v} \right\|_{L^2(\mathbb{R}^n)}^2  \le  2\pi    \mathsf{C}(\phi,\rho )  \int  \int_{0}^{1}   \frac{1}{|1+t|^{2(n-\ell)-n+1}} |\eta_j| \left|\widehat{\phi}(t \eta) \right| \left|\widehat{v}(\eta)\right|^2 \diff t \diff \eta .
\end{alignat*}
Thus, applying \eqref{aux2} and Lemma~\ref{lem:Lemma2.3Duran}, we get
\begin{alignat*}{1}
  \left\| \widehat{Q_\partial^{\ell,L}v} \right\|_{L^2(\mathbb{R}^n)}^2  \le    c_{n+1-2\ell}^2   \mathsf{C}(\phi,\rho )^2  \left\|\widehat{v} \right\|_{L^2(\mathbb{R}^n)}^2. 
\end{alignat*}
The result follows from Plancherel's theorem. 
\end{proof}

\subsection{Bound for  $Q_\partial^{\ell,U}$}

We now bound $Q_\partial^{\ell,U}$. This is the operator where we need to argue differently depending on the size of $\ell$.

\begin{lemma}[$L^p(\Omega)$--estimate]
\label{Aflemma}
Consider $p \ge 1$ such that $n/p-\ell+1>0$. If $v \in L^p(\mathbb{R}^n)$ and  is supported in $\bar\Omega$, then
\begin{equation}\label{Afp}
  \| Q_{\partial}^{\ell,U} v\|_{L^p(\Omega)}  \le       \frac{2^{\ell-n/p}}{n/p-\ell+1}   \|\partial_j \phi\|_{L^1(\Omega)} \|v\|_{L^p(\Omega)}.
\end{equation}
Moreover, assuming that  $m \ge 0$ then
\begin{equation}\label{Afinfty}
  \|Q_{\partial}^{\ell,U} v \|_{L^\infty(\Omega)}  \le      2^{\ell} |\Omega|\, \|\partial_j\phi\|_{L^\infty(\Omega)} \|v\|_{L^\infty(\Omega)}.
\end{equation}
\end{lemma}
\begin{proof}
We first prove \eqref{Afp}.  We can write
\begin{equation}\label{aux3}
  Q_{\partial}^{\ell,U} v(x)=  \int_{1/2}^1  (1-s)^{n-\ell} \int  \partial_j\phi \left(y + \frac{x-y}{s} \right) v(y) \diff y \frac{\diff s}{s^{n+1}}.
\end{equation}
We do the change of variables $z=y+\tfrac{x-y}{s}$ then we get 
\begin{equation}
 | Q_{\partial}^{\ell,U} v(x)| \le 2  \int \int_{\frac{1}{2}}^1  |1-s|^{-\ell}  |\partial_{j} \phi(z)| \left|v \left(\frac{sz-x}{s-1} \right) \right| \diff s \diff z.
\end{equation}
If we raise to power $p$, integrate, and use the Minkowski's inequality for integrals we get
\begin{equation*}
  \|Q_{\partial}^{\ell,U} v\|_{L^p(\Omega)} \le  2 \int \int_{\frac{1}{2}}^1  |1-s|^{-\ell} |\partial_{j}\phi(z)| \left(\int \left|v\left( \frac{sz-x}{s-1} \right) \right|^p  \diff x\right)^{1/p} \diff s \diff z. 
\end{equation*}
Then, applying the change of variables $\bar{x}=\tfrac{sz-x}{1-s}$  we get 
\begin{equation*}
 \|Q_{\partial}^{\ell,U} v\|_{L^p(\Omega)} \le  2 \int \int_{\frac{1}{2}}^1  |1-s|^{-\ell+n/p} |\partial_{j}\phi(z)| \left(\int   |v(\bar{x})|^p  \diff \bar{x} \right)^{1/p}  \diff s \diff z
\end{equation*}
Hence, we get 
\begin{alignat*}{1}
 \|Q_{\partial}^{\ell,U} v\|_{L^p(\Omega)} \le &  2 \|v\|_{L^p(\Omega)} \int \int_{\frac{1}{2}}^1  |1-s|^{-\ell+\frac{n}{p}} |\partial_{j}\phi(z)|  \diff s \diff z \\ 
 = &  2\|v\|_{L^p(\Omega)} \|\partial_{j}\phi\|_{L^1(B)} \int_{\frac{1}{2}}^1  |1-s|^{n/p-\ell}   \diff s \\
 =&  \frac{2^{\ell-n/p}}{n/p-\ell+1}  \|v\|_{L^p(\Omega)} \|\partial_{j} \phi\|_{L^1(B)}.
\end{alignat*}
The result \eqref{Afp} follows by applying the triangle inequality. 

Inequality  \eqref{Afinfty} easily follows from \eqref{aux3} and using that  $\int_{\frac{1}{2}}^{1}  \frac{|s-1|^{n-\ell}}{s^{n+1}} \diff s\le 2^{\ell}$.
\end{proof}

We can now apply the Riesz--Thorin interpolation theorem in conjunction with Lemma \ref{Aflemma} to obtain an $L^2$--bound. Again we must distinguish two cases depending on the size of $\ell$.
 
\begin{corollary}[estimate for $Q_{\partial}^{\ell,U}$]
\label{corAf}
 Let $v \in L^2(\Omega)$. If $n/2+1 > \ell $, then
 \begin{equation}\label{Af2-1}
  \left\|Q_{\partial}^{\ell,U} v \right\|_{L^2(\Omega)}  \le   \frac{2^{\ell-n/2}}{n/2-\ell+1}     \|\partial_j \phi\|_{L^1(\Omega)}  \|v\|_{L^2(\Omega)}.
\end{equation}
 On the other hand, if  $n/2+1 \le \ell $, then for every $p < \tfrac{n}{\ell-1} \le 2$  it holds
\begin{equation}\label{Af2}
  \|Q_{\partial}^{\ell,U} v\|_{L^2(\Omega)}  \le   \left(\frac{2^{\ell-n/p}}{n/p-\ell+1}   \|\partial_j \phi\|_{L^1(\Omega)} \right)^{\frac{p}2}  \left(2^{\ell} |\Omega| \|\partial_j\phi\|_{L^\infty(\Omega)} \right)^{1-\frac{p}2}   \|v\|_{L^2(\Omega)}.
\end{equation}
\end{corollary}
\begin{proof}
We repeat the arguments that were used in the proofs of Corollary~\ref{cor:BoundTLellbig} and Proposition~\ref{prop:estimasmallk}. For brevity we skip the details.
\end{proof}

We can now prove the main result concerning the estimates on $\partial_j Q_{2}^\ell$.

\begin{lemma}[bound on $\partial_j Q_{2}^\ell $]
Let $f \in L^2(\Omega)$. If  $n/2+1 > \ell $ then
\begin{equation}\label{Q21}
  \|\partial_j Q_{2}^\ell f\|_{L^2(\Omega)} \le C(n, \ell) \|f\|_{L^2(\Omega)}.
\end{equation}
On the other hand, if  $n/2+1 \le \ell $, then
\begin{equation}\label{Q22}
\|\partial_j Q_{2}^\ell f\|_{L^2(\Omega)} \le  C(n,\ell) \left[1+  \left(\log \left(\frac{|\Omega|}{|B|} \right) \right)^{\frac{n}{2(\ell-1)}} \left(\frac{|\Omega|}{|B|} \right)^{\frac{2(\ell-1)-n}{2(\ell-1)}} \right] \|f \|_{L^2(\Omega)}. 
\end{equation}
\end{lemma}
\begin{proof}
First assume that  $n/2+1 > \ell $  and using \eqref{QL} and \eqref{Af2-1} we obtain
\begin{equation*}
\|\partial_j Q_{2}^\ell f\|_{L^2(\Omega)} \le  \left( c_{n+1-2\ell} \mathsf{C}(\phi_2,\rho )+ \frac{2^{\ell-n/2}}{n/2-\ell+1}     \|\partial_j \phi_2\|_{L^1(\Omega)}   \right)  \|f\|_{L^2(\Omega)}
\end{equation*}
Then we can easily show that   $\mathsf{C}(\phi_2,\rho )+ \|\partial_j \phi\|_{L^1(\Omega)}  \le C(n)$ and so \eqref{Q21} holds. 

Next, assume that $n/2+1 \le \ell $ and using \eqref{QL} and \eqref{Af2} we get
\begin{equation*}
  \|\partial_j Q_{2}^\ell f\|_{L^2(\Omega)} \le   C(n,\ell) \left[ 1+ \left(\frac{2^{\ell-n/p}}{n/p-\ell+1} \right)^{p/2}  \left( 2^{\ell} \frac{|\Omega|}{|B|} \right)^{1-p/2}   \right]  \| f \|_{L^2(\Omega)}.
\end{equation*}
Here we also used that  $\|\partial_j\phi_2\|_{L^\infty(\Omega)} \le \frac{C(n)}{|B|}$ and that $\mathsf{C}(\phi_2,\rho ) \le C(n)$. 
We write 
\begin{equation}
1-p/2=\frac{1}{2}\left(\frac{n}{\ell-1}-p \right) + \left(1-\frac{n}{2(\ell-1)} \right),
\end{equation}
and choose $p=\frac{n}{\ell-1}-\epsilon$ where $\epsilon=\frac{2}{\log \left(\frac{|\Omega|}{|B|}\right)}$. Hence, 
\begin{equation*}
\frac{1}{2}\left(\frac{n}{\ell-1}-p \right)= \frac{1}{\log \left(\frac{|\Omega|}{|B|}\right)}
\end{equation*}
We thus have that 
\begin{equation*}
  \left(2^{\ell} \frac{|\Omega|}{|B|} \right)^{1-p/2} \le 2^{\ell( 1 - p/2 )} e \left(\frac{|\Omega|}{|B|} \right)^{\frac{2(\ell-1)-n}{2(\ell-1)}}\le C(n, \ell)  \left(\frac{|\Omega|}{|B|} \right)^{\frac{2(\ell-1)-n}{2(\ell-1)}}.
\end{equation*}

Also, we notice that
\begin{equation*}
\frac{1}{n/p-\ell+1}=\frac{p}{n-p(\ell-1)}=\frac{p}{\epsilon (\ell-1)}=\frac{p \log \left(\frac{|\Omega|}{|B|} \right)}{2(\ell-1)},
\end{equation*}
which allows us to estimate
\begin{equation*}
 \left(\frac{2^{\ell-n/p}}{n/p-\ell+1} \right)^{p/2} \le \left(\frac{p}{2(\ell-1)}\right)^{p/2} \left(\log \left(\frac{|\Omega|}{|B|} \right)\right)^{p/2} \le C(n, \ell) \left(\log \left(\frac{|\Omega|}{|B|} \right) \right)^{\frac{n}{2(\ell-1)}}. 
\end{equation*}

This concludes the proof.
\end{proof}

We now bound the term involving $\partial_j Q_{1}^\ell $.

\begin{lemma}[bound on $\partial_j Q_{1}^\ell $]
Let $f \in L^2(\Omega)$ and $g(y)=yf(y)$. If  $n/2+1 > \ell $ then
\begin{equation}\label{Q11}
\|\partial_j Q_{1}^\ell g\|_{L^2(\Omega)} \le C(n, \ell) \frac{R}{\rho} \|f\|_{L^2(\Omega)}.
\end{equation}
On the other hand, if  $n/2+1 \le \ell $, then
\begin{equation}\label{Q12}
\|\partial_j Q_{1}^\ell g\|_{L^2(\Omega)} \le  C(n,\ell) \frac{R}{\rho}\left[1+  \left(\log \left(\frac{|\Omega|}{|B|} \right) \right)^{\frac{n}{2(\ell-1)}} \left(\frac{|\Omega|}{|B|} \right)^{\frac{2(\ell-1)-n}{2(\ell-1)}} \right] \|f\|_{L^2(\Omega)}. 
\end{equation}
\end{lemma}
\begin{proof}
First assume that  $n/2+1 > \ell $  and using \eqref{QL} and \eqref{Af2-1} we obtain
\begin{equation*}
\|\partial_j Q_{1}^\ell g\|_{L^2(\Omega)} \le   \left(c_{n+1-2\ell} \mathsf{C}(\phi_1,\rho )+ \frac{2^{\ell-n/2}}{n/2-\ell+1}     \|\partial_j \phi_1\|_{L^1(\Omega)}   \right)  \|g\|_{L^2(\Omega)}
\end{equation*}
Then we can easily show that   $\mathsf{C}(\phi_1,\rho )+ \|\partial_j \phi_1\|_{L^1(\Omega)}  \le \frac{C(n)}{\rho}$ and since $f$ is supported in $\Omega$ we have $\|g\|_{L^2(\Omega)} \le R \|f\|_{L^2(\Omega)}$. This proves \eqref{Q11}. 

Next, assume that $n/2+1 \le \ell $ and using \eqref{QL} and \eqref{Af2} we get
\begin{equation*}
\|\partial_j Q_{1}^\ell g\|_{L^2(\Omega)} \le   \frac{C(n,\ell)}{\rho} \left[ 1+ \left(\frac{2^{\ell-n/p}}{n/p-\ell+1} \right)^{p/2}  \left(2^{\ell} \frac{|\Omega|}{|B|} \right)^{1-p/2}  \right]  \|g\|_{L^2(\Omega)}.
\end{equation*}
Here we also used that  $\|\partial_j\phi_1\|_{L^\infty(\Omega)} \le \frac{C(n)}{\rho |B|}$,  $\|\partial_j\phi_1\|_{L^\infty(\Omega)} \le \frac{C(n)}{\rho}$ and that $\mathsf{C}(\phi_1,\rho ) \le \frac{C(n)}{\rho}$. Proceeding as we did in the proof \eqref{Q22}, and using that  $\|g\|_{L^2(\Omega)} \le R \|f\|_{L^2(\Omega)}$, proves \eqref{Q12}. 
\end{proof}

We are ready to prove the main estimate regarding the components that comprise the Bogovski\u{\i}--type operator $\Bog$.

\begin{theorem}[bound on $Q_\ell$]
\label{thm:boudnQ}
Let $\Omega$ be a bounded domain that is star shaped with respect to a ball $B$. Set $R=\diam(\Omega)$, and $\rho = \diam(B)$. Then, for $\ell \in \{1, \ldots, n \}$, the operator $Q_\ell$, defined in \eqref{eq:BogReduced}, satisfies
\[
  | Q_\ell f |_{H^1(\Omega)} \leq C(n, \ell)  \frac{R}\rho \kappa \| f \|_{L^2(\Omega)},
\]
where $C(n,\ell)$ is a constant that only depends on $n$ and $\ell$, and $\kappa = \kappa( \Omega, B, R, \rho)$ is such that,
\begin{enumerate}[1.]
  \item If $\ell < n/2+1$, then
  \[
    \kappa = 1.
  \]
  
  \item If  $\ell \ge n/2+1 $, then
  \[
    \kappa = 1+  \left(\log \left(\frac{|\Omega|}{|B|} \right) \right)^{\frac{n}{2(\ell-1)}} \left(\frac{|\Omega|}{|B|} \right)^{\frac{2(\ell-1)-n}{2(\ell-1)}} .
  \]
\end{enumerate}
\end{theorem}
\begin{proof}
It suffices to gather the previously obtained estimates for and $\partial_j Q_{2}^\ell f$ and $\partial_j Q_{1}^\ell g$.
\end{proof}

As a consequence we obtain the second main result of this work. An estimate on the continuity constant for the Bogovski\u{\i}--type operators $\Bog$.

\begin{corollary}[estimate on $C_{\Bog,1}$]
\label{cor:BogEstimate}
Let $\Omega$ be a bounded domain that is star shaped with respect to a ball $B$. Set $R=\diam(\Omega)$, and $\rho = \diam(B)$. Then, for $\ell \in \{1, \ldots, n \}$, the operator $\Bog$, defined in \eqref{eq:defBogovskii}, satisfies
\[
  | \Bog u |_{H^1(\Omega,\Lambda^{\ell-1})} \leq C(n, \ell)  \frac{R}\rho \kappa \| u \|_{L^2(\Omega, \Lambda^\ell)},
\]
where $C(n,\ell)$ is a constant that only depends on $n$ and $\ell$ and $\kappa = \kappa( \Omega, B, R, \rho)$ is as in Theorem~\ref{thm:boudnQ}.
\end{corollary}

\section{A chain of star shaped domains}
\label{sec:LipschitzDomains}

Let us now extend the technique to estimate the constants in our operators to more general domains. To do so we will follow some ideas presented in \cite{MR3198867} to decompose domains but in a much simpler setting. Our techniques for the case of no boundary conditions were also inspired by the proof of the so--called Mayer--Vietoris theorem as presented in \cite[Theorem 15.9]{MR1930091}.

Let us begin by presenting the class of domains $\Omega$ to which our results shall apply. Essentially, we will deal with a chain of domains over which the estimate can be extended. We will assume that $\Omega$ is a contractible domain, such that there is $N \in \polN$ for which
\[
  \Omega = \bigcup_{i=1}^N \Omega_i
\]
where:
\begin{enumerate}[$\bullet$]
  \item For each $i \in \{1, \ldots, N\}$ the domain $\Omega_i$ is star shaped with respect to a ball $B_i \subset \Omega_i$.
  
  \item For every $i,j \in \{ 1, \ldots, N\}$ with $|i-j|>1$ we have $\Omega_i \cap \Omega_j = \emptyset$.
  
  \item For $i \in \{1, \ldots, N-1\}$ let $\Omega_{i+1/2} = \Omega_i \cap \Omega_{i+1} \neq \emptyset$. Then $\Omega_{i+1/2}$ is star shaped with respect to a ball $B_{i+1/2} \subset \Omega_{i+1/2}$.
  
  \item We have a partition of unity subject to this decomposition. In other words, there are $\{\phi_i\}_{i=1}^N  \subset C^\infty(\overline{\Omega})$, such that $0 \le \phi_i \le 1$, $\phi_i =0$ in $\Omega \setminus \Omega_i$ and $\sum_{i=1}^N \phi_i=1$ in $\overline{\Omega}$.
   
 \item Finally, we impose a restriction on the way the sets can intersect, in the sense that for $i \in \{1, \ldots, N\}$ and any multiindex $\alpha \in \polN_0^n$
  \[
    \| \partial^\alpha \phi_i \|_{L^\infty(\Omega)} \leq \frac{C_{\alpha,i}}{d_i^{|\alpha|}},
  \]
  where $d_i = \min \{\diam (\Omega_{i-1/2}), \diam (\Omega_{i+1/2}) \}$. We comment that this last assumption is common in the domain decomposition literature; see \cite[Assumptions 3.1, 3.2]{toselli2006domain}.
\end{enumerate}
Notice that the conditions of our decomposition guarantee that, for every  $x \in \Omega$,
\[
  1 \leq \# \{ i : x \in \Omega_i \} \leq 2.
\]

The estimates of Corollary~\ref{cor:PoincareEstimate} and \ref{cor:BogEstimate} depend on the geometric characteristics of the domain. Specifically on the ratio of the diameter of the domain and the ball, and the ratio of their measures. Let us denote by $C_D$ the constant in these estimates for a domain $D$. Then we set
\[
  \polT = \left\{1, \ldots, N \right\} \cup \left\{ \frac32, \frac52, \ldots, N - \frac12 \right\},
\]
and
\begin{align}
\label{eq:defofCtree}
  C_\polT &= \max\left\{ C_{\Omega_t} : t \in \polT \right\}, \\
\label{eq:defofBigDtree}
  D_\polT &= \max\left\{ \diam(\Omega_t) : t \in \polT \right\}, \\
\label{eq:defofLittledtree}
  d_\polT &= \min\left\{ \diam(\Omega_t) : t \in \polT \right\}, \\
\label{eq:defofCSeparation}
  C_S &= \max\left\{ C_{\alpha,i} : |\alpha| \leq 2, i \in \{1, \ldots, N-1\} \right\}.
\end{align}

We also need to recall Poincar\'e's inequality as stated, for example, in \cite[equation (7.44)]{gilbarg2015elliptic}.

\begin{lemma}[Poincar\'e inequality I]
\label{lem:PoincI}
Let $t \in \polT$, $\ell \in \{0, \ldots, n\}$, $v \in H_0^1(\Omega_t,\Lambda^{\ell})$. Then we have that 
\begin{equation*}
\|v\|_{L^2(\Omega_t)} \le \mathsf{C}(n) \diam(\Omega_t)  |v|_{H^1(\Omega_t,\Lambda^{\ell})},
\end{equation*}
 where the constant $\mathsf{C}(n)$ only depends on $n$. 
\end{lemma}

The next result is well known, and is also sometimes referred to as Poincar\'e's inequality. Of importance to us here is an estimate on the value of the constant. This result, in the language of vector fields, was presented in \cite[Section 5]{MR3086804}. For completeness, we provide a proof.

\begin{lemma}[Poincar\'e inequality II]
Let $t \in \polT$, and $u \in H^1(\Omega_t,\Lambda^{0})$ be such that $\int_{\Omega_t} \star u =0$. Then, there is a constant $K_{P_t}$ such that
\begin{equation*}
\|u\|_{L^2(\Omega_t,\Lambda^0)} \le K_{P_t} \diam(\Omega_t)  |u|_{H^1(\Omega_t,\Lambda^0)}.
\end{equation*}
Moreover, the constant $K_{P_t}$ can be bounded by
\[
  K_{P_t} \leq C \frac{R_t}{\rho_t}\left[1+  \left(\log \left(\frac{|\Omega_t|}{|B_t|} \right) \right)^{\frac{n}{2(n-1)}} \left(\frac{|\Omega_t|}{|B_t|} \right)^{\frac{n-2}{2(n-1)}} \right],
\]
where the constant $C$ depends only on the dimension $n$, $R_t = \diam (\Omega_t)$ and $\rho_t=\diam(B_t)$.
\end{lemma}
\begin{proof}
Since $\star u \in L^2(\Omega_t, \Lambda^n)$ has zero average, we can deduce from Corollary~\ref{cor:BogEstimate} (with $\ell=n$) the existence of $v \in H^1_0(\Omega,\Lambda^{n-1})$ such that $\diff v = \star u$ and, since $\star$ is an isometry,
\begin{equation}
\label{eq:Poincarezeroavgforms}
  | v |_{H^1(\Omega_t, \Lambda^{n-1})} \leq C \frac{R_t}{\rho_t} \left[1+  \left(\log \left(\frac{|\Omega_t|}{|B_t|} \right) \right)^{\frac{n}{2(n-1)}} \left(\frac{|\Omega_t|}{|B_t|} \right)^{\frac{n-2}{2(n-1)}} \right] \| u \|_{L^2(\Omega_t,\Lambda^0)},
\end{equation}
with a constant $C$ that depends only on the dimension. Now,
\begin{align*}
  \| u \|_{L^2(\Omega_t,\Lambda^0)}^2 &= \int_{\Omega_t} u \wedge \star u = \left| \int_{\Omega_t} u \wedge \diff v \right|= \left| \int_{\Omega_t} \diff u \wedge  v \right| \\ &\leq C | u |_{H^1(\Omega_t,\Lambda^0)} \| v \|_{L^2(\Omega_t,\Lambda^{n-1})},
\end{align*}
where the constant $C$ depends only on the dimension. Since $v \in H^1_0(\Omega,\Lambda^{n-1})$, using Lemma~\ref{lem:PoincI} and estimate \eqref{eq:Poincarezeroavgforms} the result follows.
\end{proof}

Having realized that the constants $K_{P_i}$ can be bound, once again, by geometric characteristcs of our domains, we set
\begin{equation}
\label{eq:valueofCP}
  C_P=\max\left\{  K_{P_t} : t \in \polT \right\}.
\end{equation}

\subsection{Using the Poincar\'e operator}
Our result about estimating the continuity constant for more general domains then reads as follows. Interestingly, value of the constant is independent of $N$. This, of course, provided the value of the constants defined in \eqref{eq:defofCtree}---\eqref{eq:defofCSeparation} and \eqref{eq:valueofCP} is also independent of $N$.

\begin{theorem}[estimate on a chain: without boundary conditions]
Let $\Omega$ satisfy all the previously stated conditions, $\ell \in \{1, \ldots, n\}$ and $u \in L^2(\Omega, \Lambda^\ell)$ be such that $\diff u = 0$. Then, there is $v \in H^1(\Omega, \Lambda^{\ell-1})$ such that $\diff v = u$ and, moreover,
\[
  |v|_{H^1(\Omega,\Lambda^{\ell-1})} \leq C(C_\polT, D_\polT, d_\polT, C_S ) \| u \|_{L^2(\Omega,\Lambda^\ell)}.
\]
An upper bound for the constant in this estimate is given by
\[
  C(C_\polT, D_\polT, d_\polT, C_S ) \leq 2C_\polT \begin{dcases}
                                                     \sqrt{ 1 + 32 C_S^2 \left( \frac{ C_\polT C_P D_\polT}{d_\polT} + 1 \right)^4 }, & \ell \geq 2, \\
                                                    1, & \ell = 1.
                                                   \end{dcases}
\]
\end{theorem}
\begin{proof}
Let, for the time being, $\ell \geq 2$. Since, by assumption, all the $\{\Omega_i\}_{i=1}^N$ are star shaped with respect to a ball, a combination of Theorem~\ref{thm:Costabel} and Corollary~\ref{cor:PoincareEstimate} yield the existence of $\eta_i \in H^1(\Omega_i,\Lambda^{\ell-1})$ such that, in their domain of definition $\diff \eta_i = u$, and
\[
  | \eta_i |_{H^1(\Omega_i,\Lambda^{\ell-1})} \leq C_\polT \| u \|_{L^2(\Omega_i,\Lambda^\ell)}.
\]
Notice that we can add and subtract a suitable constant to $\eta_i$ to conclude, via Poincar\'e inequality, that
\[
  \| \eta_i \|_{L^2(\Omega_i, \Lambda^{\ell-1})} \leq C_P D_\polT | \eta_i |_{H^1(\Omega_i,\Lambda^{\ell-1})} \leq C_P D_\polT C_\polT \| u \|_{L^2(\Omega_i,\Lambda^\ell)}.
\]
While this provides a solution to the problem locally, the issue at hand is that, for $i \in \{1, \ldots, N - 1 \}$, $\eta_i, \eta_{i+1}$ may not coincide on the intersection $\Omega_{i+1/2}$. Thus, we must make a local correction.

Let $i \in \{1, \ldots, N-1\}$ and notice that, 
\begin{equation}
\label{eq:diffetasizero}
  \diff (\eta_i - \eta_{i+1} ) = 0 \quad \text{ on } \Omega_{i+1/2}.
\end{equation}
Since by assumption $\Omega_{i+1/2}$ is star shaped with respect to a ball, we can apply again Theorem~\ref{thm:Costabel} and Corollary~\ref{cor:PoincareEstimate}  to find $w_{i+1/2} \in H^1(\Omega_{i+1/2}, \Lambda^{\ell-2})$ such that
\[
  \diff w_{i+1/2} = \eta_i - \eta_{i+1} \quad \text{ on } \Omega_{i+1/2}.
\]
and
\begin{align*}
  \| w_{i+1/2} \|_{L^2(\Omega_{i+1/2}, \Lambda^{\ell-2})} &\leq C_P D_\polT | w_{i+1/2} |_{H^1(\Omega_{i+1/2},\Lambda^{\ell-2})} \\
  &\leq C_P D_\polT C_\polT \| \eta_i - \eta_{i+1} \|_{L^2(\Omega_{i+1/2},\Lambda^{\ell-1})} \\
  &\leq 2C_P^2 D_\polT^2 C_\polT^2 \| u \|_{L^2(\Omega_i \cup \Omega_{i+1},\Lambda^{\ell})}.
\end{align*}
Here we used Poincar\'e's inequality twice. Set $\phi_0 \equiv 0 \equiv \phi_{N+1}$ and $w_{-1/2} \equiv 0 \equiv w_{N+1/2}$. We then define, for $1 \le i \le N$,
\begin{equation*}
v_i=\eta_i+ \diff(\phi_{i-1} w_{i-1/2}- \phi_{i+1} w_{i+1/2})  \quad \text{ in } \Omega_{i}.
\end{equation*}
We see that $\diff v_i= \diff\eta_i=u$ in $\Omega_i$. Moreover, in $\Omega_{i+1/2}$,
\begin{align*}
  v_{i+1}-v_{i}=&\big(\eta_{i+1}+\diff(\phi_{i} w_{i+1/2})\big)-\big( \eta_{i}-\diff(\phi_{i+1} w_{i+1/2})\big) \\
  = &\big(\eta_{i+1}-\eta_i)+\diff((\phi_{i}+\phi_{i+1}) w_{i+1/2})
  =  \big(\eta_{i+1}-\eta_i)+\diff(w_{i+1/2})
  =0.   
\end{align*}
Here we used that $\phi_{i+2}$ and $\phi_{i-1}$ vanish on $\Omega_{i+1/2}$ and that $\phi_i+\phi_{i+1}=1$ on $\Omega_{i+1/2}$. 
Consequently, we can define $v \in H^1(\Omega, \Lambda^{\ell-1})$ by $v|_{\Omega_i} = v_i$ for every $i$. We also have
\[
  \diff v = \diff v_i = u.
\]

It remains then to provide a bound on the seminorm of $v$. To this end, 
\begin{alignat*}{1}
  |v|_{H^1(\Omega,\Lambda^{\ell-1})}^2 \leq & \sum_{i=1}^N |v_i|_{H^1(\Omega_i,\Lambda^{\ell-1})}^2\\
   \leq & 2 \sum_{i=1}^N |\eta_i |_{H^1(\Omega_i,\Lambda^{\ell-1})}^2 + 2 \sum_{i=1}^{N} |\diff (\phi_{i-1} w_{i-1/2}-\phi_{i+1} w_{i+1/2}) |_{H^1(\Omega_{i},\Lambda^{\ell-1})}^2.
\end{alignat*}

Since every point $x\in \Omega$ belongs to at most two subsets
\begin{equation}
\label{eq:glob_estimate_eta_part}
  \sum_{i=1}^N |\eta_i |_{H^1(\Omega_i,\Lambda^{\ell-1})}^2 \leq C_\polT^2 \sum_{i=1}^N \| u \|_{L^2(\Omega_i,\Lambda^\ell)}^2 \leq 2 C_\polT^2 \| u \|_{L^2(\Omega,\Lambda^\ell)}^2.
\end{equation}

We also have
\begin{equation*}
 2 \sum_{i=1}^{N} |\diff (\phi_{i-1} w_{i-1/2}-\phi_{i+1} w_{i+1/2}) |_{H^1(\Omega_{i},\Lambda^{\ell-1})}^2 \le  8 \sum_{i=1}^{N} |\diff (\phi_{i+1} w_{i+1/2}) |_{H^1(\Omega_{i+1/2},\Lambda^{\ell-1})}^2. 
\end{equation*}

Now, on every $\Omega_{i+1/2}$,

\begin{align*}
  | \diff( \phi_{i+1} w_{i+1/2} ) |_{H^1(\Omega_{i+1/2},\Lambda^{\ell-1})} =& \left| \phi_{i+1} \diff w_{i+1/2} + \diff \phi_{i+1} \wedge w_{i+1/2}\right|_{H^1(\Omega_{i+1/2},\Lambda^{\ell-1})} \\
    \leq & | \phi_{i+1} (\eta_i - \eta_{i+1}) |_{H^1(\Omega_{i+1/2},\Lambda^{\ell-1})} \\ 
    &+ | \diff \phi_{i+1} \wedge w_{i+1/2} |_{H^1(\Omega_{i+1/2},\Lambda^{\ell-1})},
\end{align*}
with
\begin{align*}
  | \phi_{i+1} (\eta_i - \eta_{i+1}) |_{H^1(\Omega_{i+1/2},\Lambda^{\ell-1})} \leq & C_S | \eta_i - \eta_{i+1}|_{H^1(\Omega_{i+1/2},\Lambda^{\ell-1})} \\ &+ \frac{C_S}{d_\polT} \| \eta_i - \eta_{i+1} \|_{L^2(\Omega_{i+1/2},\Lambda^{\ell-1})} \\
  \leq & 2C_S C_\polT \left( 1 + \frac{C_P D_\polT}{d_\polT} \right) \| u \|_{L^2(\Omega_i \cup \Omega_{i+1}, \Lambda^\ell)},
\end{align*}
and
\begin{multline*}
  | \diff \phi_{i+1} \wedge w_{i+1/2} |_{H^1(\Omega_{i+1/2},\Lambda^{\ell-1})}  \\
  \leq \frac{C_S}{d_\polT^2} \| w_{i+1/2} \|_{L^2(\Omega_{i+1/2},\Lambda^{\ell-1})} + \frac{C_S}{d_\polT} | w_{i+1/2} |_{H^1(\Omega_{i+1/2},\Lambda^{\ell-1})} \\
  \leq \frac{2C_SC_PD_\polT C_\polT^2}{d_\polT}\left( \frac{ C_P D_\polT}{d_\polT} + 1 \right) \| u \|_{L^2(\Omega_i \cup \Omega_{i+1}, \Lambda^\ell)}.
\end{multline*}
Therefore, using that $C_\polT \geq 1$,
\[
  | \diff( \phi_{i+1/2} w_{i+1/2} ) |_{H^1(\Omega_{i+1/2},\Lambda^{\ell-1})} \leq 2 C_S C_\polT \left( \frac{ C_\polT C_P D_\polT}{d_\polT} + 1 \right)^2 \| u \|_{L^2(\Omega_i \cup \Omega_{i+1}, \Lambda^\ell)}.
\]

Using, once again, that every point $x\in \Omega$ belongs to at most two subsets
\begin{equation}
\label{eq:dwestglobal}
\begin{aligned}
  &8 \sum_{i=1}^{N} |\diff (\phi_{i+1} w_{i+1/2}) |_{H^1(\Omega_{i+1/2},\Lambda^{\ell-1})}^2 \\
  \leq & 16C_S^2 C_\polT^2 \left( \frac{ C_\polT C_P D_\polT}{d_\polT} + 1 \right)^4 \sum_{i=1}^{N-1}\| u \|_{L^2(\Omega_i \cup \Omega_{i+1}, \Lambda^\ell)}^2 \\
   \leq & 64 C_S^2 C_\polT^2 \left( \frac{ C_\polT C_P D_\polT}{d_\polT} + 1 \right)^4\| u \|_{L^2(\Omega, \Lambda^\ell)}^2.
\end{aligned}
\end{equation}

Gathering \eqref{eq:glob_estimate_eta_part} and \eqref{eq:dwestglobal}
\[
  |v|_{H^1(\Omega,\Lambda^{\ell-1})}^2 \leq 2 \left( 2C_\polT^2 +  64 C_S^2 C_\polT^2 \left( \frac{ C_\polT C_P D_\polT}{d_\polT} + 1 \right)^4 \right) \| u \|_{L^2(\Omega, \Lambda^\ell)}^2,
\]
which is the claimed estimate in the case $\ell \ge 2$. 

Now let us turn to the case $\ell=1$. In this case,  \eqref{eq:diffetasizero} implies that $\eta_{i}-\eta_{i+1}$ is a constant  on $\Omega_{i+1/2}$ which we denote by $b_i$. Hence, we define constants $c_i$ recursively satisfying
\begin{equation}\label{ci}
c_{i+1}=c_i+b_{i+1},
\end{equation}
with $c_1=0$. Then we set $v_i=\eta_i+c_i$ on $\Omega_i$. We see that $\diff v_i= \diff \eta_i=u$ on $\Omega_i$. Moreover, $v_{i+1}-v_{i}=(\eta_{i+1}+c_{i+1}) -(\eta_{i}+c_i)=0$ on $\Omega_{i+1/2}$. Therefore, we define $v \in H^1(\Omega, \Lambda^{0})$ by $v|_{\Omega_i} = v_i$ for every $i$. In this case, we have
\begin{alignat*}{1}
  |v|_{H^1(\Omega,\Lambda^{0})}^2 \leq & \sum_{i=1}^N |v_i|_{H^1(\Omega_i,\Lambda^{0})}^2\leq  2 \sum_{i=1}^N |\eta_i |_{H^1(\Omega_i,\Lambda^{0})}^2. 
\end{alignat*}
Combining this with \eqref{eq:glob_estimate_eta_part} gives the estimate in the case $\ell=1$. 
\end{proof}

\begin{remark}[$L^2$--bounds]
Although we only focused on an estimate for the $H^1$--seminorm we could have also obtained an estimate for the $L^2$--norm. The $L^2$ estimate in the case of $\ell \in \{2, \ldots, n\}$ would have followed easily. In the case $\ell=1$, however, the estimate would not have been so well behaved since the constants in \eqref{ci} would have more dependence on each other. As a consequence, at least with our technique, the constant in the $L^2$--bound would depend linearly on $N$ when $\ell =1$.
\end{remark}

\subsection{Using the Bogovski{\u{\i}} operator}
In this section we use the Bogovski{\u{\i}} operator, which we estimated in  Corollary \ref{cor:BogEstimate}, to prove estimates on a chain or star shaped domains for functions that have boundary conditions. We shall only consider the cases $\ell \in \{1,\ldots, n-1\}$. The result for $\ell=n$ was proved in \cite[Corollary 3.1]{MR3198867} in much greater generality (e.g. $L^p$ norms and allowing the cardinality of the subdomains $\{\Omega_i\}$ in the decomposition to be countable and having more overlap). In the fact, the result $\ell=n$ is very special in the sense that the constant seems to be worst behaved. One key difference is that in the case $\ell=n$ one has to correct forms to make them have average zero. In contrast, in the case $1 \le \ell < n$, one has to correct the forms to make them have vanishing exterior derivative.
 
For this reason, we need to provide an additional condition to our decomposition. From the previous assumptions we already had that $\phi_i$ vanishes on $\partial\Omega_i \setminus \partial\Omega$. We make this slightly stronger as follows:
\begin{enumerate}[$\bullet$]
\item The function $\phi_i$ vanishes in a neighborhood of $\partial\Omega_i \setminus \partial\Omega$.
\end{enumerate}

We then begin with an auxiliary result. The proof of this result is presented in Appendix~\ref{app:proof_lemB}.

\begin{lemma}[vanishing trace]
\label{lem:auxiliarylemmaBogovskiichain}
Let $i \in \{1, \ldots, N-1\}$, $\ell \in \{1, \ldots, n-1\}$, and $u \in L^2(\Omega_{i+1/2},\Lambda^\ell)$ be such that $\diff u = 0$. Define $w = \diff(\phi_{i+1}u)$, then $\tr_{\partial\Omega_{i+1/2}} w = 0$.
\end{lemma}

\begin{theorem}[estimate on a chain: with boundary conditions]\label{thembog}
Let $\Omega$ satisfy all the previously stated conditions, $\ell \in \{1, \ldots, n-1\}$ and $u \in L^2(\Omega, \Lambda^\ell)$ be such that $\diff u = 0$ and  $\tr_{\partial\Omega}u = 0$. Then, there is $v \in H_0^1(\Omega, \Lambda^{\ell-1})$ such that $\diff v = u$ and, moreover,
\[
  |v|_{H^1(\Omega,\Lambda^{\ell-1})} \leq C(C_\polT, D_\polT, d_\polT, C_S ) \| u \|_{L^2(\Omega,\Lambda^\ell)}.
\]
An upper bound for the constant in this estimate is given by
\begin{equation*}
  C(C_\polT, D_\polT, d_\polT, C_S ) \leq  4  C_\polT \sqrt{2+ C_S^2\mathsf{C}(n)^2 \frac{D_\polT^2}{d_\polT^2} C_\polT^2}.
\end{equation*}
\end{theorem}
\begin{proof}
By Lemma~\ref{lem:auxiliarylemmaBogovskiichain} we have that $\tr_{\partial\Omega_{i+1/2}}\diff(\phi_{i+1} u)$. 

Define $W_i = \cup_{1 \leq j \leq i} \Omega_j$, and let $\widetilde{\phi}_{i+1} =0$ be the function that coincides with $\phi_{i+1}$ in $\Omega_{i+1/2}$, equals zero in $W_i\setminus \Omega_{i+1/2}$ and equals one in $\Omega \setminus W_i$. The additional assumption we imposed in the decomposition guarantees that $\widetilde{\phi}_{i+1} \in C^1(\overline{\Omega})$.
Then, since $\diff u = 0$ in $\Omega$, 
\[
  \int_{\Omega_{i+1/2}} \diff(\phi_{i+1}u) = \int_{\Omega} \diff(\widetilde{\phi}_{i+1}u) 
  = \langle \tr_{\partial\Omega} u , \widetilde{\phi}_{i+1} \rangle = 0.
\]

Hence,  using Theorem~\ref{thm:Costabel} and Corollary~\ref{cor:BogEstimate} we obtain $w_{i+1/2} \in H_0^1(\Omega_{i+1/2},\Lambda^{\ell})$,   such that $\diff w_{i+1/2}= \diff(\phi_{i+1} u)$ on $\Omega_{i+1/2}$ with the estimate
\begin{multline*}
  \| w_{i+1/2} \|_{L^2(\Omega_{i+1/2}, \Lambda^{\ell})} \leq \mathsf{C}(n) D_\polT | w_{i+1/2} |_{H^1(\Omega_{i+1/2},\Lambda^{\ell})} \\
    \leq \mathsf{C}(n) D_\polT C_\polT \| \diff(\phi_{i+1} u) \|_{L^2(\Omega_{i+1/2},\Lambda^{\ell+1})}
  \leq C_S\mathsf{C}(n) \frac{D_\polT}{d_\polT} C_\polT \| u \|_{L^2(\Omega_{i+1/2},\Lambda^{\ell})}.
\end{multline*}

We thus have 
 \begin{equation*}
 \diff(\phi_i u)=\diff (w_{i-1/2}-w_{i+1/2}) \quad \text{ in } \Omega_i.
 \end{equation*}
 Here we used that $-\diff(w_{i+1/2})=-\diff(\phi_{i+1} u)=\diff(\phi_i u)$ in $\Omega_{i+1/2}$.  We also used that $\phi_i=1$ on $\Omega_i \backslash (\Omega_{i-1/2} \cup \Omega_{i+1/2})$.  Again, using Theorem~\ref{thm:Costabel} and Corollary~\ref{cor:BogEstimate} we can find $v_i \in H_0^1(\Omega_{i},\Lambda^{\ell-1})$ such that $\diff v_i= \phi_i u +w_{i+1/2}-w_{i-1/2}$. With the bound
 \begin{equation*}
  | v_i |_{H^1(\Omega_{i},\Lambda^{\ell})} \le  C_\polT \| \phi_i u +w_{i+1/2}-w_{i-1/2}\|_{L^2(\Omega_{i+1/2}, \Lambda^{\ell})}.
 \end{equation*}
  We then define $v=\sum_{i=1}^N v_i \in H_0^1(\Omega,\Lambda^{\ell-1})$ and see that 
 \begin{equation*} 
 \diff  v= \sum_{i=1}^N \diff  v_i= \sum_{i=1}^N \phi_i u=u \quad \text{ on } \Omega.
 \end{equation*}
 
 Moreover,
\begin{alignat*}{1}
  |v|_{H^1(\Omega,\Lambda^{\ell-1})}^2 \leq & 2 \sum_{i=1}^N |v_i|_{H^1(\Omega_i,\Lambda^{\ell-1})}^2\\
   \leq & 16 C_\polT^2 \sum_{i=1}^N \big(\|\phi_i u \|_{L^2(\Omega_i,\Lambda^{\ell})}^2+\|w_{i+1/2}\|_{L^2(\Omega_{i+1/2},\Lambda^{\ell})}^2\big)\\ 
   \leq & 16 C_\polT^2 \big(2+ C_S^2\mathsf{C}(n)^2 \frac{D_\polT^2}{d_\polT^2} C_\polT^2\big) \|u \|_{L^2(\Omega,\Lambda^{\ell})}^2,
\end{alignat*}
which gives the desired estimate.
\end{proof}

\appendix
\section{An alternative proof of Lemma~\ref{lem:boundPthetaU}}

For diversity in our arguments let us show a direct proof of Lemma~\ref{lem:boundPthetaU}. The change of variables $z = x + \tfrac{y-x}{1-s}$ gives that
\[
 |P_\theta^{k,U} f(x) | \leq \int_{1/2}^1 s^{k-1}\int \theta(z) |f(sx+(1-s)z)| \diff z \diff s,
\]
so that
\begin{align*}
 \| P_\theta^{k,U} f \|_{L^2(\Omega)} &\leq \int_{1/2}^1 s^{k-1}\int \theta(z) \left( \int_\Omega |f(sx+(1-s)z)|^2 \diff x \right)^{1/2} \diff z \diff s \\
   &= \int_{1/2}^1 s^{k-1-n/2} \int \theta(z) \left( \int_\Omega |f(\bar x)|^2 \diff \bar x\right)^{1/2} \diff z \diff s,
\end{align*}
where we again used that $\Omega$ is star shaped with respect to a ball, so that if $x \in \Omega$ and $z \in \supp \theta \subset B$, then $\bar x = sx+(1-s)z \in \Omega$. With this technique then the same estimate can be concluded.

\section{Proof of Lemma~\ref{lem:auxiliarylemmaBogovskiichain}}
\label{app:proof_lemB}
Intuitively it is clear that the result holds. To see this, we observe that can write  $\partial \Omega_{i+1/2}=\Gamma_0 \sqcup \Gamma_1 \sqcup \Gamma_2$ where $\Gamma_0=\partial \Omega \cap \partial \Omega_{i+1/2}$, $\Gamma_1= (\partial \Omega_{i+1/2} \cap \partial \Omega_{i}) \backslash \Gamma_0$ and  $\Gamma_2= (\partial \Omega_{i+1/2} \cap \partial \Omega_{i+1}) \backslash \Gamma_0$. Note that  $\phi_{i+1} \equiv 1$ on $\Gamma_1$ and $\phi_{i+1} \equiv 0$ on $\Gamma_2$. Thus, we have that $\tr_{\partial \Omega_{i+1/2}}\diff \phi_{i+1}=0$ in $\Gamma_1 \cup \Gamma_2$. We now note that, since  $\diff u=0$, we have $\diff(\phi_{i+1} u)= \diff \phi_{i+1} \wedge u$ and as a consequence $\tr_{\partial \Omega_{i+1/2}}\diff(\phi_{i+1} u)=\tr_{\partial \Omega_{i+1/2}}( \diff \phi_{i+1} \wedge u)= \tr_{\partial \Omega_{i+1/2}}\diff \phi_{i+1} \wedge \tr_{\partial \Omega_{i+1/2}}u$. Here we used that the trace operator (or more generally a pullback) respects the wedge product \cite[Lemma 14.16 (b)]{MR1930091}. Since $\tr_{\partial \Omega_{i+1/2}} u =0$ on $\Gamma_0$ and we showed that $\tr_{\partial \Omega_{i+1/2}}\diff \phi_{i+1}=0$ on $\Gamma_1 \cup \Gamma_2$, we conclude that $\tr_{\partial \Omega_{i+1/2}} \diff(\phi_{i+1} u)=0$.

Let us now be more rigorous in our reasoning.

\begin{proof}
Notice, first of all, that since $w \in L^2(\Omega_{i+1/2}, \Lambda^\ell)$ and $\diff w = 0$, the trace $\tr_{\Omega_{i+1/2}} w$ is well defined.

By definition, if $\psi \in H^1(\Omega_{i+1/2}, \Lambda^{n-\ell-2})$,
\begin{align*}
  \left|\langle \tr_{\Omega_{i+1/2}} w, \psi \rangle\right| &= \left|\int_{\Omega_{i+1/2}} w \wedge \diff \psi \right|= 
  \left| \int_{\Omega_{i+1/2}} u \wedge \diff \phi_{i+1} \wedge \diff \psi \right| \\
  &= \left| \int_\Omega u \wedge \diff \widetilde{\phi}_{i+1} \wedge \diff \psi \right|,
\end{align*}
where we used that $\diff u = 0$ in $\Omega_{i+1/2}$, and $\widetilde{\phi}_{i+1}$ has the same meaning as in the proof of Theorem~\ref{thembog}.

Observe now that $\diff \widetilde{\phi}_{i+1} \wedge \diff \psi = - \diff(\diff \widetilde{\phi}_{i+1} \wedge \psi )$ so that, invoking the fact that $\diff u = 0$ in $\Omega$, we obtain
\[
  \left|\langle \tr_{\Omega_{i+1/2}} w, \psi \rangle\right| = \left|\langle \tr_{\Omega} u, \diff \widetilde{\phi}_{i+1} \wedge \psi \rangle\right| =0,
\]
where in the last step we used that $\diff \widetilde{\phi}_{i+1} \wedge \psi \in H^1(\Omega, \Lambda^{n-\ell-1})$.
\end{proof}

\bibliographystyle{plain}
\bibliography{biblio}

\end{document}